\documentclass[11pt,a4paper]{article}
\usepackage{amssymb}
\usepackage{eurosym}
\usepackage{amsfonts}
\usepackage{amsmath}
\usepackage{amsthm}
\usepackage{graphicx}
\usepackage{float}
\usepackage{hyperref}

\hypersetup{
	colorlinks=true,
	linkcolor=blue,
	anchorcolor=blue,
	citecolor=blue
}
\newtheorem{theorem}{Theorem}[section]

\newtheorem{corollary}[theorem]{Corollary}
\newtheorem{lemma}[theorem]{Lemma}
\newtheorem{proposition}[theorem]{Proposition}
\theoremstyle{definition}
\newtheorem{definition}[theorem]{Definition}

\newtheorem{remark}[theorem]{Remark}

\renewenvironment{proof}[1][Proof]{\noindent\textbf{#1.} }{\ \rule{0.5em}{0.5em}}
\newenvironment{acknowledgement}{\smallskip{\sc Acknowledgement.}\rm}{\smallskip}
\renewcommand{\theequation}{\thesection.\arabic{equation}}
\allowdisplaybreaks
\input{tcilatex}

\def\func#1{\mathop{\mathrm{#1}}\nolimits}

\def\dint{\displaystyle\int}

\def\Xint#1{\mathchoice
{\XXint\displaystyle\textstyle{#1}}%
{\XXint\textstyle\scriptstyle{#1}}%
{\XXint\scriptstyle\scriptscriptstyle{#1}}%
{\XXint\scriptscriptstyle\scriptscriptstyle{#1}}%
\!\int}
\def\XXint#1#2#3{{\setbox0=\hbox{$#1{#2#3}{\int}$ }
\vcenter{\hbox{$#2#3$ }}\kern-.6\wd0}}

\def\oint{\Xint-}

\def\enddoc{\end{document}}

\def\Qlb#1{#1}

\def\Qcb#1{#1} 

\def\FRAME#1#2#3#4#5#6#7#8
{
 \begin{figure}[H]
 \begin{center}
 \includegraphics[height=#3]{#7}
 \caption{#5}
 \label{#6}
 \end{center}
 \end{figure}
}

\begin{document}
\title{Sharp sub-Gaussian upper bounds for subsolutions of Trudinger's equation on Riemannian manifolds}
\author{Philipp S\"urig}
\date{February 2024}
\maketitle

\begin{abstract}
We consider on Riemannian manifolds the nonlinear evolution equation
\begin{equation*}
\partial _{t}u=\Delta _{p}(u^{1/(p-1)}),
\end{equation*}%
where $p>1$. This equation is also known as a doubly non-linear parabolic equation or Trudinger's equation. We prove that weak subsolutions of this equation have a sub-Gaussian upper bound and prove that this upper bound is sharp for a specific class of manifolds including $\mathbb{R}^{n}$.
\end{abstract}

\tableofcontents

\let\thefootnote\relax\footnotetext{
Funded by the Deutsche Forschungsgemeinschaft (DFG, German Research	Foundation) - Project-ID 317210226 - SFB 1283.}

\section{Introduction}	

We are concerned here with solutions of the non-linear evolution equation 
\begin{equation}
\partial _{t}u=\Delta _{p}u^{q},  \label{evoeq}
\end{equation}
where $p>1$, $q>0$, $u=u(x,t)$ is an unknown non-negative function and $%
\Delta _{p}$ is the $p$-Laplacian 
\begin{equation*}
\Delta _{p}v=\func{div}\left( |\nabla v|^{p-2}\nabla v\right) .
\end{equation*}
The equation (\ref{evoeq}) is frequently referred to as a \emph{doubly non-linear parabolic equation}.

G. I. Barenblatt constructed in \cite{barenblatt1952self} spherically symmetric self-similar solutions of (\ref{evoeq}) in $\mathbb{R}^{n}$, that are nowadays called \textit{Barenblatt solutions}. In the case \begin{equation}\label{qp-q}q(p-1)=1\end{equation}
the Barenblatt solution is given by 
\begin{equation}\label{barenborder}
u(x,t)=\frac{1}{t^{n/p}}\exp \left( -\zeta \left( \frac{|x|}{t^{1/p}}\right)
^{\frac{p}{p-1}}\right) ,
\end{equation}%
where $\zeta =(p-1)^{2}p^{-\frac{p}{p-1}}$.
In this case the equation (\ref{evoeq}) becomes as follows \begin{equation}\label{evoeq1}\partial _{t}u=\Delta _{p}\left(u^{\frac{1}{p-1}}\right),\end{equation} and is referred to as \textit{Trudinger's equation} (\cite{trudinger1968pointwise}). For $p=2$ we get $q=1$ so that equation (\ref{evoeq1}) becomes the classical \textit{heat equation} $\partial _{t}u=\Delta u$, and the function (\ref{barenborder}) becomes the classical \textit{heat kernel} or \textit{Gauss-Weierstrass} function. 

The Trudinger equation (\ref{evoeq1}) in $\mathbb{R}^{n}$, was investigated in \cite{del2004nonlinear, grillo2006super, ishige1996existence}, where existence and uniqueness results and decay properties were proved. 

In the present paper we are interested in estimates for solutions of
the Trudinger equation (\ref{evoeq1}) on geodesically complete Riemannian manifolds. We understand solutions in a certain weak sense (see Subsection \ref{secweakdef} for the definition). For the heat equation on Riemannian manifolds, that is, when $p=2$, we refer to \cite{Coulhon1997, grigor1994heat, Grigoryan2012,  grigor2009heat, grigor2016surgery, grigor2022volume} for related results.

In the case $q(p-1)>1$, solutions of (\ref{evoeq}) on Riemannian manifolds have the property of \textit{finite propagation speed} which was investigated in \cite{andreucci2021asymptotic, bonforte2005asymptotics,  de2022wasserstein,  dekkers2005finite, grigor2023finite, Grigor’yan2024, grillo2016smoothing, vazquez2015fundamental}. Solutions of (\ref{evoeq}) in the case $q(p-1)<1$ on Riemannian manifolds were, for example, considered in \cite{andreucci2021extinction, bonforte2008fast}. 

Let $M$ be a geodesically complete Riemannian manifold. Denote by $\mu$ the \textit{Riemannian measure} on $M$, by $d$ the \textit{geodesic distance} and by $B(x, r)$ the \textit{geodesic ball} of radius $r$ centered at $x$.

One of the first main results of the present paper (cf. \textbf{Theorem \ref{corcurv}}) is as follows.
\begin{theorem}\label{mainthmint}
Let $M$ satisfy a \textit{relative Faber-Krahn inequality} (see Subsection \ref{secMoser}). Let $u$ be a bounded non-negative solution to (\ref{evoeq1}) in $M\times [0, \infty)$ with an initial function $u_{0}=u(\cdot ,0$) and set $A=\textnormal{supp}~u_{0}$. Then, for all $x\in M$ and all $t>0$, \begin{equation}\label{mainre}||u(\cdot,t)||_{L^{\infty}\left(B(x,\frac{1}{2}t^{1/p})\right)}\leq\frac{C}{\mu(B(x, t^{1/p}))}\exp\left(-c\left(\frac{d(x, A)}{t^{1/p}}\right)^{\frac{p}{p-1}}\right),\end{equation} where $C,c$ are positive constants.
\end{theorem}

For example, the relative Faber-Krahn inequality is satisfied if $M$ has non-negative Ricci curvature (see \cite{Buser, grigor, Saloff}).

Comparing the estimate from Theorem \ref{mainthmint} with the Barenblatt solution (\ref{barenborder}), we see that the estimate  (\ref{mainre}) is sharp in $\mathbb{R}^{n}$. Moreover, we not only get a sharp upper bound in $\mathbb{R}^{n}$, but also in some class of spherically symmetric manifolds (model manifolds) satisfying the relative Faber-Krahn inequality (cf. Proposition \ref{supersolmodel}).

Recall that in the linear case $p=2$ the classical Li-Yau estimate (\cite{Li1986OnTP}) says the following: if $M$ has non-negative Ricci curvature, the heat kernel $p_{t}(x, y)$ (that is the fundamental solution of the heat equation)
satisfies for all $x, y\in M$ and all $t>0$, \begin{equation}\label{liyau}p_{t}(x, y)\asymp \frac{C}{\mu(B(x, \sqrt{t}))}\exp\left(-c\frac{d^{2}(x,y)}{t}\right),\end{equation} where "$\asymp$" means that an upper and lower bound hold, where the corresponding constants might differ. If $M$ satisfies a relative Faber-Krahn inequality then the upper bound in (\ref{liyau}) was proved in \cite{grigor1994heat}.
Clearly the upper bound in (\ref{liyau}) matches (\ref{mainre}) with $p=2$.

The second main result of the present paper is the following (cf. \textbf{Theorem \ref{upperCH}}).

\begin{theorem}\label{upperCHint}
Let $M$ satisfy a \textit{uniform Sobolev inequality} (see Subsection \ref{secMoser}) and $n$ be the dimension of $M$. Let $u$ be a bounded non-negative solution to (\ref{evoeq1}) in $M\times [0, \infty)$ with an initial function $u_{0}=u(\cdot ,0$) and set $A=\textnormal{supp}~u_{0}$. Then, for all $x\in M$ and all large enough $t>0$, \begin{equation}\label{cartharint}||u(\cdot,t)||_{L^{\infty}\left(B(x,\frac{1}{2}t^{1/p})\right)}\leq\frac{C}{t^{n/p}}\exp\left(-c\left(\frac{d(x,A)}{t^{1/p}}\right)^{\frac{p}{p-1}}\right),\end{equation} where $C, c$ are positive constants.
\end{theorem}
The uniform Sobolev inequality holds, for example, if $M$ is a \textit{Cartan-Hadamard manifold} (see \cite{hoffman1974sobolev}).

It is quite surprising that similar estimates to (\ref{mainthmint}) and (\ref{cartharint}) hold for heat kernels on \textit{fractal spaces}. In this case $n$ denotes the \textit{fractal dimension} of the space and $p$ plays a role of the \textit{walk dimension} (see \cite{barlow2006diffusions}).

The main technical difficulties of proving the main results Theorem \ref{mainthmint} and Theorem \ref{upperCHint} arise due to the interplay between the non-linearity of (\ref{evoeq1}) and the geometry of the underlying space. These difficulties primarily appear in the proof of the following main lemma, which is worth mentioning here (cf. Lemma \ref{Gaussup}). 
\begin{lemma}\label{maintechlem}
Let $u$ be a bounded non-negative solution to (\ref{evoeq1}) in $M\times [0, \infty)$. Set $A=\textnormal{supp}~u_{0}$ and $\rho=c\max\left(d(x, A), Ct^{1/p}\right)$, where $c,C$ are positive constants. Denote with $A_{\rho}$ the $\rho$-neighborhood of $A$, that is, $A_{\rho}=\{x\in M:d(x, A)<\rho\}$. Let $\gamma$ be a \textit{regular function} (see Subsection \ref{daviesty} for the definition) satisfying, for fixed large enough $\lambda$ and for all $s>0$, \begin{equation}\label{gammaintro}\int_{A_{\rho}}u^{\lambda}(\cdot, s)\leq \frac{1}{\gamma(s)}.\end{equation} Then, for all $x\in M$ and all $t>0$, \begin{equation}\label{uxTupperint}||u(\cdot,t)||_{L^{\infty}\left(B(x,\frac{1}{2}t^{1/p})\right)}\leq\left(\frac{C_{B}}{\mu(B\left( x,t^{1/p}\right))\gamma(t)}\right)^{\frac{1}{\lambda}}\exp\left(-c^{\prime}\left(\frac{d(x,A)}{t^{1/p}}\right)^{\frac{p}{p-1}}\right),\end{equation} where $c^{\prime}, C_{B}>0$ and $C_{B}$ depends on the intrinsic geometry of $B\left( x,t^{1/p}\right)$.
\end{lemma}
Let us emphasize that this lemma is valid for an arbitrary complete Riemannian manifold, and the estimate (\ref{uxTupperint}) depends on the \emph{local} structure of the manifold inside the ball $B\left( x,t^{1/p}\right)$.
The rather mild assumptions on the global geometry of $M$ in Theorem \ref{mainthmint} and Theorem \ref{upperCHint} then enable us to compute the function $\gamma$ from (\ref{gammaintro}) and obtain the estimates (\ref{mainre}) and (\ref{cartharint}).

Let us describe the structure of the present paper.

In Section \ref{secweak} we define the notion of weak solutions of the equation (\ref{evoeq}).

In Section \ref{intestsec} we prove in Lemma \ref{integdec} (Subsection \ref{intmaxp}) an \textit{integral maximum principle} for subsolutions of (\ref{evoeq1}). The idea of the integral maximum principle goes back to Aronson in \cite{Aronson1967BoundsFT, aronson1968non}. In the present setting, this principle says the following. Let $u$ be a non-negative bounded subsolution of (\ref{evoeq1}), $\lambda>0$ be large enough and $\xi(x, t)$ be a non-positive locally Lipschitz function in $M\times[0, \infty)$ satisfying the differential inequality \begin{equation*}\partial_{t}\xi+C_{p, \lambda}|\nabla \xi|^{p}\leq 0,\end{equation*} where $C_{p, \lambda}$ is a positive constant.
Then the integral \begin{equation*}J(t)=\int_{M}u^{\lambda}(\cdot, t)e^{\xi(\cdot, t)}\end{equation*} is non-increasing in $t>0$.
The proof utilizes the in our case specific shape of the \textit{Caccioppoli type inequality} (Lemma \ref{Caccioppoli}). 

For a specific choice of $\xi$ (see Lemma \ref{choicofxi}), we obtain in Subsection \ref{daviesty} a \textit{Davies-Gaffney type inequality} (Corollary \ref{CorD}). This inequality gives, for large enough $\lambda>0$, a \textit{Sub-Gaussian} upper bound for the $L^{\lambda}$-norm of non-negative bounded subsolutions of (\ref{evoeq1}) and is the first ingredient for obtaining the estimate (\ref{uxTupperint}) of Lemma \ref{maintechlem}.

In Section \ref{secmvi} we prove the second ingredient $-$ a $L^{\lambda}$-\textit{mean value inequality} for non-negative bounded subsolutions of (\ref{evoeq1}) (Lemma \ref{LemMoser}). The proof essentially follows the method from \cite{grigor2023finite}, where particular case of a mean value inequality was proved in the case $q(p-1)\geq 1$ by using the classical Moser iteration argument (\cite{Moser}).

In Section \ref{secgauss} we finally prove Lemma \ref{maintechlem}, by combining the Davies-Gaffney type inequality and the mean value inequality (see Lemma \ref{Gaussup}). In the linear heat equation this method of combining these two inequalities was used by Davies in \cite{davies1992heat} as well as it was used in \cite{grigor1999estimates} and \cite{coulhon1998random} to prove an estimate of the type (\ref{liyau}).

In Section \ref{appendixla} (Appendix) we discuss the aforementioned model manifold.

We denote by $c, c^{\prime}, C, C^{\prime}$ positive constants whose value might change at each occurance. For functions $f$ and $g$ we also use the notation $f\simeq g$ if there exists a positive constant $C$ such that $C^{-1}g\leq f\leq Cg$. Similarly, we define the symbol $\lesssim$.

\begin{acknowledgement}
The author would like to thank Alexander Grigor'yan for many helpful discussions.
\end{acknowledgement}

\section{Weak subsolutions}
\label{secweak}
\subsection{Definition and basic properties}
\label{secweakdef}

We consider in what follows the following evolution equation on a Riemannian
manifold $M$:%
\begin{equation}
\partial _{t}u=\Delta _{p}u^{q}.  \label{olddtv}
\end{equation}%
By a \textit{subsolution} of (\ref{olddtv}) we mean a non-negative function $u$
satisfying 
\begin{equation}
\partial _{t}u\leq \Delta _{p}u^{q}.\label{subdtv}
\end{equation}%
in a certain weak sense as explained below. 

We assume throughout that 
\begin{equation*}
p>1\ \ \text{and}\ \ \ q>0.
\end{equation*}%

Let $\Omega$ be an open subset of $M$ and $I$ be an interval in $[0, \infty)$.

\begin{definition}
\normalfont
We say that a non-negative function $u=u(x, t)$ is a \textit{weak
subsolution} of (\ref{olddtv}) in $\Omega\times I$, if
\begin{equation}  \label{defvonsoluq}
u\in C\left(I; L^{1}(\Omega)\right)\cap
\left\{u^{q}\in L_{loc}^{p}\left(I; W^{1, p}(\Omega)\right)\right\}
\end{equation}
and (\ref{subdtv}) holds weakly in $\Omega\times I$, which means that for all $t_{1}, t_{2}\in I$ with $t_{1}<t_{2}$, and all non-negative functions 
\begin{equation}  \label{defvontestsoluq}
\psi\in W_{loc}^{1, \infty}\left(I;
L^{\infty}(\Omega)\right)\cap L_{loc}^{p}\left(I; W_{0}^{1,
p}(\Omega)\right),
\end{equation}
we have 
\begin{equation}  \label{defvonweaksolq}
\left[\int_{\Omega}{u\psi}\right]_{t_{1}}^{t_{2}}+\int_{t_{1}}^{t_{2}}{%
\int_{\Omega}{-u\partial_{t}\psi+|\nabla u^{q}|^{p-2}\langle\nabla u^{q},
\nabla \psi\rangle}}\leq 0.
\end{equation}
\end{definition}

For different notions of weak solutions see also \cite{dibenedetto2011harnack, sturm2017existence}. For existence and uniqueness results for the Cauchy problem for (\ref{olddtv}) with the above notion of weak solutions, see for example in the euclidean setting \cite{ishige1996existence, ivanov1997regularity, ladyzhenskaya1968linear, raviart1970resolution} and on Riemannian manifolds \cite{grillo2018porous}.

If $u$ is of the class (\ref{defvonsoluq}), we define 
\begin{equation*}
\nabla u:=\left\{ 
\begin{array}{ll}
q^{-1}u^{1-q}\nabla(u^{q}), & u>0, \\ 
0, & u=0.%
\end{array}%
\right.
\end{equation*}

\begin{remark}
\normalfont
Note that it follows from (\ref{defvonsoluq}) and (\ref{defvontestsoluq})
that the integrals in (\ref{defvonweaksolq}) are finite. Indeed, we have by
H\"older's inequality 
\begin{align*}
\int_{t_{1}}^{t_{2}}{\int_{\Omega}{|\nabla u^{q}|^{p-2}\left|\langle\nabla u^{q},
\nabla \psi\rangle\right|}}&\leq \int_{t_{1}}^{t_{2}}{\int_{\Omega}{|\nabla
u^{q}|^{p-1}|\nabla \psi|}} \\
&\leq \left(\int_{t_{1}}^{t_{2}}{\int_{\Omega}{\left(|\nabla
u^{q}|\right)^{p}}}\right)^{\frac{p-1}{p}}\left(\int_{t_{1}}^{t_{2}}{%
\int_{\Omega}{|\nabla \psi|^{p}}}\right)^{\frac{1}{p}}.
\end{align*}
\end{remark}

From now on, let us make the assumption that
\begin{equation*}\label{deltazero}
	(p-1)q-1=0,
\end{equation*}%
that is \begin{equation}\label{defq}q=\frac{1}{p-1}>0.\end{equation}
Hence, the evolution equation (\ref{olddtv}) becomes \begin{equation}\label{dtv}\partial _{t}u=\Delta _{p}\left(u^{\frac{1}{p-1}}\right).\end{equation}

Everywhere $M$ denotes a geodesically complete Riemannian manifold.
Let $\mu $ denote the Riemannian measure on $M$. For simplicity of
notation, we omit in almost all integrations the notation of measure. All
integration in $M$ is done with respect to $d\mu $, and in $M\times \mathbb{R%
}$ -- with respect to $d\mu dt$, unless otherwise specified.

\subsection{Caccioppoli type inequality}

Let $\Omega $ be an open subset of $M$ and $I$ be an interval in $[0, \infty)$.
\begin{lemma}[\cite{grigor2023finite}]\label{Caccioppoli}
\label{Lem1}Let $v=v\left( x,t\right) $ be a bounded
non-negative subsolution to \emph{(\ref{dtv})} in a cylinder $\Omega\times I$.
Let $\eta \left( x,t\right) $ be a locally Lipschitz non-negative bounded
function in $\Omega\times I$ such that $\eta \left( \cdot ,t\right) $ has
compact support in $\Omega $ for all $t\in I$. Fix some real $\lambda $
such that 
\begin{equation}\label{assumponlamb}
	\lambda \geq \max\left(p, \frac{p}{p-1}\right) 
\end{equation}%
and set 
\begin{equation*}
	\alpha =\frac{\lambda }{p}.  \label{alpha}
\end{equation*}%
Choose $t_{1},t_{2}\in I$ such that $t_{1}<t_{2}$ and set $Q=\Omega \times \left[ t_{1},t_{2}%
\right] $. Then%
\begin{equation}
\left[ \int_{\Omega }v^{\lambda }\eta ^{p}\right] _{t_{1}}^{t_{2}}+c_{1}%
\int_{Q}\left\vert \nabla \left( v^{\alpha }\eta \right) \right\vert
^{p}\leq \int_{Q}\left[ p\eta ^{p-1}\partial _{t}\eta
+c_{2}\left\vert \nabla \eta \right\vert ^{p}\right]v^{\lambda } ,
\label{veta1}
\end{equation}%
where $c_{1},c_{2}$ are positive constants depending on $p$ and $\lambda $ given by
\begin{equation}
c_{1}=\frac{\lambda \left( \lambda -1\right) }{2^{p}(p-1)^{p-1}}\alpha ^{-p}.  \label{c1}
\end{equation}
and
\begin{equation}
c_{2}=\frac{\lambda \left( \lambda -1\right) }{2(p-1)^{p-1}}\alpha ^{-p}+\dfrac{%
\lambda 2^{p-1}p^{p}}{\left( \lambda -1\right) ^{p-1}\left( p-1\right) ^{p-1}}.  \label{c2}
\end{equation}
\end{lemma}

Let us recall why all integrals in (\ref{veta1}) are well-defined. First, observe that by (\ref{assumponlamb}), \begin{equation}\label{lageq2}\lambda\geq \max\left(p, \frac{p}{p-1}\right)\geq 2.\end{equation} 
Since $v(\cdot, t)\in L^{1}(\Omega)$ and $v$ is bounded, it follows that for any $t\in I$, \begin{equation}\label{vlambfin}\int_{\Omega }v^{\lambda }\eta^{p}\leq \textnormal{const}~||v||^{\lambda-1}_{L^{\infty}(\Omega)}\int_{\Omega}v<\infty.\end{equation}
Also, we see that, \begin{equation}\label{finitesig}\int_{Q}\left[ p\eta ^{p-1}\partial _{t}\eta+c_{2}\left\vert \nabla \eta \right\vert ^{p}\right]v^{\lambda }\leq \textnormal{const}~ ||v||^{\lambda-\frac{p}{p-1}}_{L^{\infty}(Q)}\int_{Q}v^{\frac{p}{p-1} },\end{equation} which is finite by (\ref{defvonsoluq}). 
Using that $\alpha\geq \frac{1}{p-1}$, we get that the function $\Phi(s)=s^{\alpha(p-1)}$ is Lipschitz on any bounded interval in $[0, \infty)$. Thus, $v^{\alpha}=\Phi(v^{\frac{1}{p-1}})\in W^{1, p}(\Omega)$ and $$\left|\nabla v^{\alpha}\right|=\left|\Phi^{\prime}(v^{\frac{1}{p-1}})\nabla v^{q}\right|\leq C\left|\nabla v^{\frac{1}{p-1}}\right|,$$ whence \begin{equation}\label{valpha0}\int_{Q}\left\vert \nabla \left( v^{\alpha }\eta \right) \right\vert^{p}\leq C^{\prime}\int_{Q}\left\vert \nabla  v^{\alpha } \right\vert^{p}+v^{\alpha p}\leq C^{\prime\prime}\int_{Q}\left\vert \nabla  v^{\frac{1}{p-1} } \right\vert^{p}+v^{\lambda},\end{equation} which is finite using the same argument as in (\ref{finitesig}) since $\lambda=\alpha p$. In particular, (\ref{valpha0}) implies that \begin{equation}\label{valpha}v^{\alpha}\eta\in L_{loc}^{p}\left(I; W_{0}^{1, p}(\Omega)\right).\end{equation}

\begin{remark}
\label{Rempc1}
Note that (\ref{c1}) yields for all $\lambda\geq 2$, \begin{equation}\label{pc1}\frac{p}{c_{1}}=\frac{\lambda^{p}q^{1-p}2^{p}}{\lambda \left( \lambda -1\right) p^{p-1}}\leq C_{p}\lambda^{p-2},\end{equation}
where $C_{p}$ depends only on $p$ but does not depend on $\lambda $.
Therefore, if $p\leq 2$, the ratio $\frac{p}{c_{1}}$ is bounded by a constant $C_{p}$ that depends on $p$ but does not depend on $\lambda$. If $p>2$, then $\frac{p}{c_{1}}$ grows with $\lambda$ as in (\ref{pc1}).
\end{remark}

\begin{remark}
\label{Remc1c2}For the future we also need the ratio $\frac{c_{2}}{c_{1}}$. It
follows from (\ref{c1}) and (\ref{c2}) that%
\begin{align*}
\frac{c_{2}}{c_{1}}=2^{p-1}+\dfrac{2^{2p-1}\lambda ^{p}}{\left( \lambda -1\right) ^{p}},
\end{align*}%
where we have used that $\alpha p=\lambda $. It follows that, for all $\lambda \geq 2$, 
\begin{equation*}
\dfrac{c_{2}}{c_{1}}\leq C_{p},
\end{equation*}%
where $C_{p}$ depends on $p$.
\end{remark}

\begin{remark}
\label{Remc2}Let us also obtain an upper bound of $c_{2}$. Using 
$\alpha =\frac{\lambda }{p}$
we get
\begin{equation*}
c_{2}=\frac{1}{2}\frac{\lambda \left( \lambda -1\right) }{\lambda
^{p}\left( p
-1\right) ^{p-1}}p^{p}+\dfrac{\lambda 2^{p-1}p^{p}}{\left( \lambda
-1\right) ^{p-1}\left( p
-1\right) ^{p-1}}.
\end{equation*}%
As $\lambda \geq 2$, it follows that 
\begin{equation}
c_{2}\leq C_{p}\lambda ^{2-p}.  \label{c2<}
\end{equation}%
Of course, if $p\geq 2$ then $c_{2}$ is uniformly bounded by a constant $%
C_{p}$ independently of $\lambda $, but if $p<2$ then $c_{2}$ may
grow with $\lambda $ as in (\ref{c2<}).
\end{remark}

\begin{lemma}
\label{monl1}\cite{Grigor’yan2024}
Let $v=v\left( x,t\right) $ be a bounded non-negative
subsolution to \emph{(\ref{dtv})} in $M\times[0, T)$. If $\lambda\geq 1$, including $\lambda=\infty$, then the function 
\begin{equation*}
t\mapsto \left\Vert v(\cdot ,t)\right\Vert _{L^{\lambda}(M)}
\end{equation*}%
is monotone decreasing in $[0, T)$.
\end{lemma}

\section{Integral estimates for subsolutions}
\label{intestsec}

Let $M$ be a connected Riemannian manifold.
Let $d$ be the geodesic distance on $M$. For any $%
x\in M$ and $r>0$, denote by $B(x,r)$ the geodesic ball of radius $r$
centered at $x$, that is,%
\begin{equation*}
	B(x,r)=\left\{ y\in M:d(x,y)<r\right\} .
\end{equation*}

\subsection{Integral maximum principle}
\label{intmaxp}

\begin{lemma}\label{integdec}
Let $u$ be a non-negative bounded subsolution of (\ref{dtv}) in $M\times [0, T)$, where $0<T\leq \infty$. Fix some $\lambda\geq \max\left(p, \frac{p}{p-1}\right)$. Let $\xi(x, t)$ be a non-positive locally Lipschitz function in $M\times[0, T)$ and assume that the partial derivative $\partial_{t}\xi$ satisfies the inequality \begin{equation}\label{condfordec}\partial_{t}\xi+c_{2}2^{p-1}p^{-p}|\nabla \xi|^{p}\leq 0,\end{equation} where $c_{2}$ is given by (\ref{c2}).
Then the function \begin{equation}\label{defJ}J(t)=\int_{M}u^{\lambda}(\cdot, t)e^{\xi(\cdot, t)}\end{equation} is non-increasing in $t\in [0, T)$.
\end{lemma}

\begin{proof}
Since $\xi$ is non-positive, we see that the integral in (\ref{defJ}) is finite using the same arguments as in (\ref{vlambfin}). 
Let $\varphi$ be a cut-off function of some open geodesic ball $B^{\prime}$ such that $\varphi$ has compact support in some larger ball $B$. Note that the balls are precompact by the completeness of $M$. Then set $\eta(x, t)=e^{\frac{\xi(x, t)}{p}}\varphi(x)$ so that $$\eta^{p}=e^{\xi}\varphi^{p}.$$

By the Caccioppoli type inequality (\ref{veta1}), we have for $0\leq t_{1}<t_{2}<T$,
\begin{equation*}
\left[ \int_{B }u^{\lambda }\eta ^{p}\right] _{t_{1}}^{t_{2}}\leq \int_{Q}\left[ p\eta ^{p-1}\partial _{t}\eta
+c_{2}\left\vert \nabla \eta \right\vert ^{p}\right]u^{\lambda },
\end{equation*} where $Q=B\times [t_{1}, t_{2}]$.
Noticing that $$|\nabla\eta |^{p}=\left|e^{\frac{\xi}{p}}\nabla \varphi+p^{-1}e^{\frac{\xi}{p}}\varphi\nabla \xi\right|^{p}\leq 2^{p-1}\left(e^{\xi}|\nabla \varphi|^{p}+p^{-p}e^{\xi}\varphi^{p}|\nabla \xi|^{p}\right)$$ and $$p\eta ^{p-1}\partial _{t}\eta=\varphi^{p}e^{\xi}\partial_{t}\xi,$$ we obtain
\begin{align}\left[ \int_{B }{u^{\lambda }e^{\xi}\varphi^{p}}\right] _{t_{1}}^{t_{2}}\leq\int_{Q}\left[\varphi^{p}e^{\xi}\partial_{t}\xi+c_{2}2^{p-1}e^{\xi}|\nabla\varphi|^{p}+c_{2}2^{p-1}p^{-p}e^{\xi}\varphi^{p}|\nabla \xi|^{p}\right]u^{\lambda}.\label{beforecond}\end{align}
Hence, it follows from (\ref{condfordec}), that $$\left[ \int_{B }{u^{\lambda }e^{\xi}\varphi^{p}}\right] _{t_{1}}^{t_{2}}\leq c_{2}2^{p-1} \int_{Q}u^{\lambda}e^{\xi}|\nabla\varphi |^{p}.$$ Finally, we get, by sending $B\to M$ and using that $\varphi\to 1$ and $|\nabla \varphi|\to 0$ as $B\to M$, $$\left[ \int_{M }{u^{\lambda }e^{\xi}}\right] _{t_{1}}^{t_{2}}\leq 0,$$ which finishes the proof.
\end{proof}

\subsection{Davies-Gaffney type inequality}
\label{daviesty}

For any subset $A$ of $M$ and any $r>0$, set $$A_{r}=\{x\in M :d(x, A)<r\}.$$ In the next lemma, we will also use $$A_{r}^{c}:=(A_{r})^{c}=\{x\in M:d(x, A)\geq r\}.$$

\begin{lemma}\label{choicofxi}
Let $u$ be a non-negative bounded subsolution of (\ref{dtv}) in $M\times [0, \infty)$ and $A\subset M$ be measurable. Suppose that again $\lambda\geq \max\left(p, \frac{p}{p-1}\right)$. Then, for all $r, t>0$, \begin{equation}\label{upperintarc} \int_{A_{r}^{c}}u^{\lambda}(\cdot, t)\leq \int_{A^{c}}u_{0}^{\lambda}+\exp\left(-\zeta\left(\frac{r}{t^{1/p}}\right)^{\frac{p}{p-1}}\right)\int_{A}u_{0}^{\lambda},\end{equation} where $u_{0}=u(\cdot, 0)$ and \begin{equation}\label{defkappa}\zeta=\frac{p-1}{2c_{2}^{\frac{1}{p-1}}}.\end{equation}
\end{lemma}

\begin{proof}
Fix some $s>t$ and define, for all $x\in M$ and $0\leq \tau< s$, the function $$\xi(x,\tau)=-\zeta\left(\frac{d(x, A_{r}^{c})}{(s-\tau)^{1/p}}\right)^{\frac{p}{p-1}}.$$
For such $\tau$, let us also define $$J(\tau)=\int_{M}u^{\lambda}(\cdot, \tau)e^{\xi(\cdot, \tau)} .$$
Note that $\xi$ is non-positive and locally Lipschitz in $M\times[0, s)$.
Using that $|\nabla d(x, A_{r}^{c})|\leq 1$, we get $$\left|\nabla \xi(x, t)\right|^{p}\leq \frac{\zeta^{p}p^{p}d(x, A_{r}^{c})^{\frac{p}{p-1}}}{(p-1)^{p}(s-\tau)^{\frac{p}{p-1}}}.$$
Since $$\partial_{\tau}\xi=-\frac{\zeta d(x, A_{r}^{c})^{\frac{p}{p-1}}}{(p-1)(s-\tau)^{\frac{p}{p-1}}},$$ the partial derivative $\partial_{\tau}\xi$ satisfies the condition $$\partial_{\tau}\xi+c_{2}2^{p-1}p^{-p}|\nabla \xi|^{p}\leq 0.$$ Hence, we obtain by Lemma \ref{integdec} that \begin{equation}\label{Jnondeczero}J(t)\leq J(0).\end{equation}
As $d(x, A_{r}^{c})\geq r$ for all $x\in A$, we have $$\xi(x, 0)\leq -\zeta\left(\frac{r}{s^{1/p}}\right)^{\frac{p}{p-1}}~ \textnormal{for all}~x\in A.$$
Together with the fact that $\xi(x, 0)\leq0$ for all $x\in M$, we therefore obtain that \begin{equation}\label{upperJzero}J(0)=\int_{A^{c}}u_{0}^{\lambda}e^{\xi(\cdot, 0)}+\int_{A}u_{0}^{\lambda}e^{\xi(\cdot, 0)}\leq \int_{A^{c}}u_{0}^{\lambda}+\exp\left(-\zeta\left(\frac{r}{s^{1/p}}\right)^{\frac{p}{p-1}}\right)\int_{A}u_{0}^{\lambda}.\end{equation}

Since $\xi(x, t)=0$ for $x\in A_{r}^{c}$, it follows that $$J(t)\geq\int_{A_{r}^{c}}u^{\lambda}(\cdot, t)e^{\xi(\cdot, t)}=\int_{A_{r}^{c}}u^{\lambda}(\cdot, t).$$
Hence, combining this with (\ref{Jnondeczero}) and (\ref{upperJzero}), we see that $$\int_{A_{r}^{c}}u^{\lambda}(\cdot, t)\leq \int_{A^{c}}u_{0}^{\lambda}+\exp\left(-\zeta\left(\frac{r}{s^{1/p}}\right)^{\frac{p}{p-1}}\right)\int_{A}u_{0}^{\lambda}.$$
Sending now $s\to t+$, we finally obtain (\ref{upperintarc}).
\end{proof}

\begin{definition}
We call a function $\gamma:[0, \infty)\to (0, \infty)$ \textit{regular} if $\gamma$ is increasing and of \textit{at most polynomial growth}. Here, the latter means that there exist $\theta>1$ and $\Theta>0$ so that for all $t>0$, \begin{equation}\label{mopolgr}f(\theta t)\leq \Theta f(t).\end{equation}
\end{definition}

%\begin{definition}
%We call a function $\gamma:[0, \infty)\to (0, \infty)$ \textit{regular} if $\gamma$ is increasing and for some $\Theta\geq 1$, $\theta>1$ and all $0<t_{1}<t_{2}<\infty$, \begin{equation}\label{defreg}\frac{\gamma(\theta t_{1})}{\gamma(t_{1})}\leq \Theta \frac{\gamma(\theta t_{2})}{\gamma(t_{2})}.\end{equation}
%\end{definition}

%\begin{example}
%If $\gamma$ satisfies the \textit{doubling condition}, that is, for some $\Theta>1$ and all $0<t<T/2$, $$\gamma(2t)\leq \Theta \gamma(t),$$ then (\ref{defreg}) holds with $\theta=2$ as $$\frac{\gamma(2 t_{1})}{\gamma(t_{1})}\leq \Theta\leq \Theta\frac{\gamma(2 t_{2})}{\gamma(t_{2})},$$ because $\gamma$ is increasing.
%\end{example}

The next lemma is inspired by a result in the linear case from \cite{grigor1997gaussian}.

\begin{lemma}\label{upperArcIntLamb}
Let $u$ be a non-negative bounded subsolution of (\ref{dtv}) in $M\times [0, \infty)$ and let $\lambda\geq \max\left(p, \frac{p}{p-1}\right)$. Fix some $\rho>0$ and let $\gamma$ be a regular function in $[0, \infty)$. Assume that for all $t>0$, \begin{equation}\label{condfordecay}\int_{A_{\rho}}u^{\lambda}(\cdot, t)\leq \frac{1}{\gamma(t)},\end{equation} where $A=\textnormal{supp}~u_{0}$.
Then, for all $t^{1/p}\lesssim\rho$, \begin{equation}\label{decayforArc}\int_{A_{\rho}^{c}}u^{\lambda}(\cdot, t)\leq\frac{C}{\gamma(t)}\exp\left(-\varepsilon\left(\frac{\rho}{t^{1/p}}\right)^{\frac{p}{p-1}}\right),\end{equation} where $\varepsilon=\varepsilon(\theta, p, \lambda)>0$ and $C=C(\Theta, \theta, p, \lambda)$.
\end{lemma}

\begin{proof}
For any $\rho\geq\rho'>0$ and $t>0$, let us set $$J_{\rho'}(t)=\int_{A_{\rho'}^{c}}u^{\lambda}(\cdot, t).$$ Then it follows from (\ref{upperintarc}) that for all $0<\rho'<\rho$ and $0<t^{\prime}<t$, $$J_{\rho}(t)\leq J_{\rho'}(t^{\prime})+\exp\left(-\zeta\left(\frac{\rho-\rho'}{(t-t^{\prime})^{1/p}}\right)^{\frac{p}{p-1}}\right)\int_{A_{\rho'}}u^{\lambda}(\cdot, t^{\prime}).$$ Together with our assumption (\ref{condfordecay}), this yields \begin{equation}\label{JRandJr}J_{\rho}(t)\leq J_{\rho'}(t^{\prime})+\frac{1}{\gamma(t^{\prime})}\exp\left(-\zeta\left(\frac{\rho-\rho'}{(t-t^{\prime})^{1/p}}\right)^{\frac{p}{p-1}}\right).\end{equation}

Let $\{\rho_{k}\}_{k=0}^{\infty}$ and $\{t_{k}\}_{k=0}^{\infty}$ be strictly decreasing sequences of reals such that $$\rho_{0}=\rho,~ \rho_{k}\downarrow 0 \quad \textnormal{and}\quad t_{0}=t,~ t_{k}\downarrow 0.$$
Then (\ref{JRandJr}) implies that for any $k\geq1$, 
\begin{equation}\label{Jrk-1Jrk}J_{\rho_{k-1}}(t_{k-1})\leq J_{\rho_{k}}(t_{k})+\frac{1}{\gamma(t_{k})}\exp\left(-\zeta\left(\frac{\rho_{k-1}-\rho_{k}}{(t_{k-1}-t_{k})^{1/p}}\right)^{\frac{p}{p-1}}\right).\end{equation}
For $k\to \infty$, we have $$J_{\rho_{k}}(t_{k})=\int_{A_{\rho_{k}}^{c}}u^{\lambda}(\cdot, t_{k})\leq \int_{A^{c}}u^{\lambda}(\cdot, t_{k})\to \int_{A^{c}}u_{0}^{\lambda}=0,$$ where we used that $u\in C([0, \infty), L^{1}(M))$ and $u_{0}=0$ in $A$.

Hence, adding up the inequalities (\ref{Jrk-1Jrk}) for all $k\geq 1$, we deduce \begin{equation}\label{IterationJr}J_{\rho}(t)\leq \sum_{k=1}^{\infty}{\frac{1}{\gamma(t_{k})}\exp\left(-\zeta\left(\frac{\rho_{k-1}-\rho_{k}}{(t_{k-1}-t_{k})^{1/p}}\right)^{\frac{p}{p-1}}\right)}.\end{equation} Let us now specify the sequences $\{\rho_{k}\}_{k=0}^{\infty}$ and $\{t_{k}\}_{k=0}^{\infty}$ as follows: $$\rho_{k}=\frac{\rho}{k+1}\quad \textnormal{and}\quad t_{k}=\frac{t}{\theta^{k}},$$ where $\theta$ is as in (\ref{mopolgr}). In particular, we have $$\rho_{k-1}-\rho_{k}=\frac{\rho}{k(k+1)}\quad \textnormal{and}\quad t_{k-1}-t_{k}=\frac{(\theta-1)t}{\theta^{k}},$$ so that $$\zeta\left(\frac{\rho_{k-1}-\rho_{k}}{(t_{k-1}-t_{k})^{1/p}}\right)^{\frac{p}{p-1}}=\zeta \frac{\theta^{\frac{k}{p-1}}}{k^{\frac{p}{p-1}}(k+1)^{\frac{p}{p-1}}(\theta-1)^{\frac{1}{p-1}}}\left(\frac{\rho}{t^{1/p}}\right)^{\frac{p}{p-1}}\geq\varepsilon (k+1)\left(\frac{\rho}{t^{1/p}}\right)^{\frac{p}{p-1}},$$ where $$\varepsilon=\varepsilon(\theta, p, \lambda)=\inf_{k\geq 1}\zeta \frac{\theta^{\frac{k}{p-1}}}{k^{\frac{p}{p-1}}(k+1)^{\frac{p}{p-1}+1}(\theta-1)^{\frac{1}{p-1}}}>0.$$

By the regularity condition (\ref{mopolgr}) of $\gamma$, we see that $$\frac{1}{\gamma(t_{k})}\leq \frac{\Theta}{\gamma(t)}.$$
Substituting this into (\ref{IterationJr}), we obtain \begin{align*} J_{\rho}(t)&\leq\frac{\Theta}{\gamma(t)}\sum_{k=1}^{\infty}\exp\left(-\varepsilon (k+1)\left(\frac{\rho}{t^{1/p}}\right)^{\frac{p}{p-1}}\right)\\&=\frac{\Theta\exp\left(-\varepsilon\left(\frac{\rho}{t^{1/p}}\right)^{\frac{p}{p-1}}\right)}{\gamma(t)}\sum_{k=1}^{\infty}\exp\left(-\varepsilon k\left(\frac{\rho}{t^{1/p}}\right)^{\frac{p}{p-1}}\right).
\end{align*}
For $t^{1/p}\lesssim\rho$, we finally conclude \begin{align*}J_{\rho}(t) \leq \frac{C\exp\left(-\varepsilon\left(\frac{\rho}{t^{1/p}}\right)^{\frac{p}{p-1}}\right)}{\gamma(t)},\end{align*}
where $C=C(\Theta, \theta, p, \lambda)$. This implies (\ref{decayforArc}) and finishes the proof.
\end{proof}

\begin{corollary}\label{CorD}
Let $B=B(x, R)$ be a geodesic ball in $M$ and $A=\textnormal{supp}~u_{0}$. Set
\begin{equation}\label{defofrho}
\rho=c\max\left(d(x, A), CR\right),
\end{equation} where $c, C$ are positive constants chosen below and let $\gamma$ be a regular function in $[0, \infty)$. Fix some $\lambda\geq \max\left(p, \frac{p}{p-1}\right)$. Assume that for all $t>0$, \begin{equation}\label{condfordecayco}\int_{A_{\rho}}u^{\lambda}(\cdot, t)\leq \frac{1}{\gamma(t)}.\end{equation} Then, for all $t\simeq R^{p}$, \begin{equation}\label{dgtype}\int_{B}u^{\lambda}(\cdot, t)\leq \frac{C^{\prime}}{\gamma(t)}\exp\left(-c^{\prime}\left(\frac{d(x, A)}{t^{1/p}}\right)^{\frac{p}{p-1}}\right),\end{equation} where $C^{\prime}=C^{\prime}(\Theta, \theta, p, \lambda)>0$ and $c^{\prime}=c^{\prime}(\theta, p, \lambda)>0$.
\end{corollary}

\begin{proof}
Consider first the case when $d(x, A)\leq CR$. Then $\rho=cCR$ and $B\subset A_{\rho}$ if we choose $C$ large enough. Hence, (\ref{dgtype}) follows from (\ref{condfordecayco}) and the monotonicity of $\gamma$. In the case when $d(x, A)>CR$, we have $\rho=cd(x, A)$. If $\rho\leq d(A, B)$, that is \begin{equation}\label{condonc}cd(x, A)\leq d(x, A)-R,\end{equation} we get $B\subset A_{\rho}^{c}$. Indeed, we can choose $c$ small enough so that (\ref{condonc}) is satisfied, whence (\ref{decayforArc}) yields $$\int_{B}u^{\lambda}(\cdot, t)\leq \frac{C^{\prime}}{\gamma(t)}\exp\left(-\varepsilon\left(\frac{\rho}{t^{1/p}}\right)^{\frac{p}{p-1}}\right),$$ which proves (\ref{dgtype}) also in this case.
\end{proof}
\FRAME{dtbpFU}{2.5607in}{1.6588in}{0pt}{\Qcb{Sets $A$ and $B$}}{\Qlb{pic1m}}{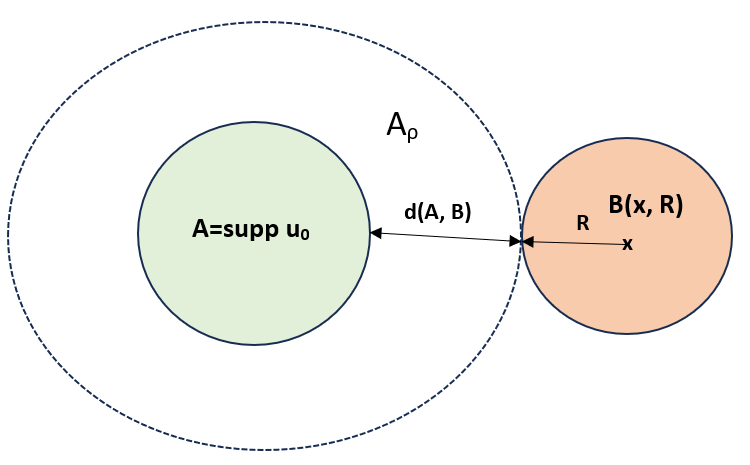}{\special{language"Scientific Word";type "GRAPHIC";maintain-aspect-ratio TRUE;display "USEDEF";valid_file "F";width 2.5607in;height 1.6588in;depth 0pt;original-width 11.078in;original-height 7.1615in;cropleft "0";croptop "1";cropright "1";cropbottom "0";filename 'pic1.png';file-properties "XNPEU";}}
%In particular, if $\mu(A), \mu(B)<\infty$, we deduce from (\ref{dgtype}), setting $g=1_{B}$, that $$\int_{B}u\leq \exp\left(-\frac{\zeta}{\lambda}\left(\frac{d(A, B)}{t^{1/p}}\right)^{\frac{p}{p-1}}\right)||u_{0}||_{L^{\infty}(A)}\left(\mu(A)\mu(B)^{\lambda-1}\right)^{1/\lambda}.$$

\section{Mean value inequality}
\label{secmvi}

\subsection{Sobolev and Faber-Krahn inequalities}
\label{secMoser}

Let $M$ be a Riemannian manifold of dimension $n$.
Recall that on geodesically complete Riemannian manifolds, every geodesic ball $B$ is precompact. Consequently, the following \textit{Sobolev inequality} in $B$ of order $%
p\geq 1$ holds: for any non-negative function 
$w\in W_{0}^{1,p}(B)$,%
\begin{equation}
	\left( \int_{B}w^{p\kappa }\right) ^{1/\kappa }\leq S_{B}\int_{B}\left\vert
	\nabla w\right\vert ^{p},  \label{SBk}
\end{equation}%
where $\kappa >1$ is some constant and $S_{B}$ is called the \emph{Sobolev
	constant} in $B$. The value of $\kappa $ is independent of $B$ and can be
chosen as follows:%
\begin{equation}
	\kappa =\left\{ 
	\begin{array}{ll}
		\dfrac{n}{n-p}, & \text{if }n>p, \\ 
		\text{any number}>1, & \text{if }n\leq p.%
	\end{array}%
	\right.  \label{k}
\end{equation}%

We always assume that $S_{B}$ is chosen to be minimal possible. In this case
the function $B\mapsto S_{B}$
is clearly monotone increasing with respect to inclusion of balls.

We say that $M$ admits a \textit{uniform Sobolev inequality} if (\ref{SBk}) holds with $S_{B}\leq \func{const}$
for all geodesic balls $B\subset M$. This holds for example, in the case when $M$ is a Cartan-Hadamard manifold (see \cite{hoffman1974sobolev}).

Let $\kappa ^{\prime }=\frac{\kappa }{\kappa -1}$
be the H\"{o}lder conjugate of $\kappa $ and set
\begin{equation}\label{nu}
	\nu =\dfrac{1}{\kappa ^{\prime }}=\left\{ 
	\begin{array}{ll}
		\dfrac{p}{n}, & \text{if }n>p, \\ 
		\text{any number}\in (0, 1), & \text{if }n\leq p.%
	\end{array}%
	\right..
\end{equation}
Denoting by $r(B)$ the radius of $B$, let us define a new quantity%
\begin{equation}\label{defiota}
	\iota (B)=\dfrac{1}{\mu (B)}\left( \dfrac{r(B)^{p}}{S_{B}}\right) ^{\frac{1}{\nu}}
\end{equation}%
so that 
$S_{B}=\frac{r(B)^{p}}{\left( \iota (B)\mu \left( B\right) \right) ^{\nu}}.$
Then (\ref{SBk}) can be written in the form%
\begin{equation}
	\int_{B}\left\vert \nabla w\right\vert ^{p}\geq 
	\dfrac{(\iota (B)\mu(B))^{\nu}}{r(B)^{p}}\left(
	\int_{B}w^{p\kappa }\right) ^{1/\kappa }.  \label{Sp}
\end{equation}%
The constant $\iota (B)$ is called the \emph{normalized Sobolev} constant in 
$B$.
It is clear from (\ref{Sp}) that the value of $\kappa $ can be always
reduced (by modifying the value of $\iota (B)$). It is only important that $%
\kappa >1$. In fact, the exact value of $\kappa $ does not affect the
results, although various constants do depend on $\kappa $.

Letting $D=\left\{ w>0\right\}$ we deduce from (\ref{Sp}) and H\"{o}lder's inequality that
\begin{equation}
	\dint_{B}\left\vert \nabla w\right\vert ^{p}\geq \dfrac{1}{r(B)^{p}}\left( \iota (B)\dfrac{\mu (B)}{\mu (D)}\right) ^{\nu
	}\dint_{B}w^{p}.   \label{FKp}
\end{equation}
We call inequality (\ref{FKp}) \textit{Faber-Krahn inequality}
of order $p\geq 1$.
We say that $M$ satisfies a \textit{relative Faber-Krahn inequality} of order $p\geq 1$, if (\ref{FKp}) holds with $\iota(B)\geq \textnormal{const}>0$ for all geodesic balls $B\subset M$. For example, this holds if $M$ is complete and satisfies $Ricci_{M}\geq 0$ (see \cite{Buser, grigor, Saloff}).

\begin{lemma}[Moser inequality]\cite{grigor2023finite}
	\label{MoserLem}Let $w\in L^{p}\left( I;W_{0}^{1,p}(B)\right)$ be
	non-negative and $I$ be an interval in $\mathbb{R}_{+}$. Set $Q=B\times I$. 
	Then
	\begin{equation}
		\dint_{Q}w^{p\left( 1+\nu \right) }\leq S_{B}\left( \dint_{Q}\left\vert
		\nabla w\right\vert ^{p}\right) \sup\limits_{t \in I}\left( \dint_{B}w^{p}\right)
		^{\nu }.  \label{Moser}
	\end{equation}
\end{lemma}

%\begin{proof}
%By the H\"{o}lder inequality, we obtain, for any $t\in I $,%
%\begin{align*}
%\int_{B}w^{p\left( 1+\nu \right) }& =\int_{B}w^{p}w^{p\nu }\leq \left(
%\int_{B}w^{p\kappa }\right) ^{1/\kappa }\left( \int_{B}w^{p\nu \kappa
	%^{\prime }}\right) ^{1/\kappa ^{\prime }} \\
%& =\left( \int_{B}w^{p\kappa }\right) ^{1/\kappa }\left(
%\int_{B}w^{p}\right) ^{\nu } \\
%& \leq \left( \int_{B}w^{p\kappa }\right) ^{1/\kappa %}\sup_{t\in I }\left( \int_{B}w^{p}\right) ^{\nu },
%\end{align*}%
%where we have used that $\nu \kappa ^{\prime }=1$.

%By the Sobolev inequality (\ref{SBk}) we deduce
%\begin{equation*}
%\left( \int_{B}w^{p\kappa }\right) ^{1/\kappa }\leq S_{B}\int_{B}\left\vert
%\nabla w\right\vert ^{p}.
%\end{equation*}%
%Hence, it follows that%
%\begin{equation*}
%\int_{B}w^{p\left( 1+\nu \right) }\leq S_{B}\left( \int_{B}\left\vert \nabla
%w\right\vert ^{p}\right) \sup_{t\in I}\left( \int_{B}w^{p}\right) ^{\nu }.
%\end{equation*}%
%Integrating this inequality in $t\in I $ implies (\ref{Moser}).
%\end{proof}

\subsection{Comparison in two cylinders}

\begin{lemma}
\label{Lemtwo}Consider two balls $B=B\left( x,r\right) $ and $B^{\prime}=B\left( x,r^{\prime }\right) $ with $0<r^{\prime }<r$, and two cylinders%
\begin{equation*}
Q=B\times \lbrack t,T],\ \ \ Q^{\prime }=B^{\prime }\times \left[ t^{\prime },T\right],\end{equation*} where $0\leq t<t^{\prime }<T$.
Let $\lambda $ be any real such that $\lambda \geq \max\left(p, \frac{p}{p-1}\right)$.
Let $v$ be a non-negative bounded subsolution of \emph{(\ref{dtv})} in $%
B\times \lbrack t,T]$.
Then 
\begin{equation}
\int_{Q^{\prime }}v^{\lambda \left( 1+\nu \right) }\leq \frac{CS_{B}\max\left(\lambda^{p-2}, \lambda^{\left( 2-p\right) \nu }\right)}{\min\left(t^{\prime }-t, (r-r^{\prime })^{p}\right) ^{ 1+\nu }} \left( \int_{Q}v^{\lambda
}\right) ^{1+\nu },  \label{vQ'QQ}
\end{equation}%
where constant $C$ depends on $p$ and $\nu $ (given by (\ref{nu})), but it is
independent of $\lambda $.
\end{lemma}

\begin{proof}
Let us consider in $Q$ the function $\eta(x, \tau)=\eta_{1}(x)\eta_{2}(\tau)$, where $\eta_{1}$ is a bump function of $B^{\prime}$ in $B$ and \begin{equation*}
\eta_{2}(\tau)=\left\{ 
\begin{array}{ll}
\frac{\tau-t}{t^{\prime}-t}, & t\leq \tau<t^{\prime}, \\ 
1, & t^{\prime}\leq \tau\leq T.
\end{array}%
\right.
\end{equation*} Set again $\alpha =\frac{\lambda }{p}$ and recall that then by (\ref{valpha}), $v^{\alpha }\eta\in L^{p}\left([t,T]; W_{0}^{1, p}(B)\right)$. Hence, applying the Moser inequality (\ref{Moser}) with $w=v^{\alpha }\eta$
and using $w^{p}=v^{\lambda }\eta ^{p},$
we obtain
\begin{equation*}
\int_{Q}v^{\lambda \left( 1+\nu \right) }\eta ^{p\left( 1+\nu \right) }\leq
S_{B}\left( \int_{Q}\left\vert \nabla \left( v^{\alpha }\eta \right)
\right\vert ^{p}\right) \sup_{\tau\in \left[ t,T\right] }\left(
\int_{B}v^{\lambda }\eta ^{p}\right) ^{\nu }.
\end{equation*}%
By (\ref{veta1}) we have%
\begin{equation*}
\int_{Q}\left\vert \nabla \left( v^{\alpha }\eta \right) \right\vert
^{p}\leq \int_{Q}\left[\frac{p}{c_{1}}\eta^{p-1}\partial_{\tau}\eta +\frac{c_{2}}{c_{1}}\left\vert \nabla \eta
\right\vert ^{p}\right]v^{\lambda }
\end{equation*}%
and 
\begin{equation*}
\sup_{\tau\in \left[ t',T\right] }\left( \int_{B}v^{\lambda }\eta ^{p}\right)
\leq \int_{Q}\left[p\eta^{p-1}\partial_{\tau}\eta +c_{2}\left\vert \nabla \eta \right\vert ^{p}\right]v^{\lambda },
\end{equation*}%
where in the latter we used that $\eta_{2}(t)=0$.
Therefore, it follows that 
\begin{equation*}
\int_{Q}v^{\lambda \left( 1+\nu \right) }\eta ^{p\left( 1+\nu \right) }\leq
S_{B}\left(\int_{Q}\left[\frac{p}{c_{1}}\eta^{p-1}\partial_{\tau}\eta +\frac{c_{2}}{c_{1}}\left\vert \nabla \eta
\right\vert ^{p}\right]v^{\lambda }\right)\left(\int_{Q}\left[p\eta^{p-1}\partial_{\tau}\eta +c_{2}\left\vert \nabla \eta
\right\vert ^{p}\right]v^{\lambda }\right)^{\nu}.
\end{equation*}%
Using that $\eta =1$ in $Q^{\prime }$, $\left\vert \nabla \eta
\right\vert \leq \frac{1}{r-r^{\prime }}$ and $\partial_{\tau}\eta\leq \frac{1}{t^{\prime }-t}$ we obtain 
\begin{align*}
\int_{Q^{\prime }}v^{\lambda \left( 1+\nu \right) }&\leq S_{B} \left[\frac{p}{c_{1}(t^{\prime }-t)} +\frac{c_{2}}{c_{1}(r-r^{\prime })^{p}}\right]\left(\int_{Q}v^{\lambda }\right)\left[\frac{p}{t^{\prime }-t}+\frac{c_{2}}{(r-r^{\prime })^{p}}\right]^{\nu}\left(\int_{Q}v^{\lambda }\right)^{\nu}\\&\leq S_{B}\frac{2^{1+\nu}\max(\frac{p}{c_{1}}, \frac{c_{2}}{c_{1}})\max(p, c_{2})^{\nu}}{\min\left(t^{\prime }-t, (r-r^{\prime })^{p}\right) ^{ 1+\nu }}\left(\int_{Q}v^{\lambda}\right)^{1+\nu}.
\end{align*}%
By Remark \ref{Rempc1} and \ref{Remc1c2}, we have $\frac{p}{c_{1}}\leq C_{p}\lambda^{p-2}$ and 
$\frac{c_{2}}{c_{1}}\leq C_{p}$.
Combining this with the estimate (\ref{c2<}) of Remark \ref{Remc2}, that is,
$c_{2}\leq C_{p}\lambda ^{2-p}$,
we conclude (\ref{vQ'QQ}).
\end{proof}

\subsection{Iterations and the mean value theorem}
\label{SecMV}

\begin{lemma}
\label{LemMoser}Let $B=B\left( x_{0},R\right) $ and $%
T>0$. Let $u$ be a non-negative bounded subsolution of \emph{(\ref{dtv})} in 
$B\times \lbrack 0,T\rbrack$.
Let us set 
\begin{equation*}\label{defQ}
Q=B\times \left[ 0,T\right] .
\end{equation*}%
Then, for the cylinder 
\begin{equation*}\label{defQp}
Q^{\prime }=\frac{1}{2}B\times \left[ 2^{-p}T,T\right] ,
\end{equation*}%
we have for all $\lambda>0$,
\begin{equation}
\left\Vert u\right\Vert _{L^{\infty }(Q^{\prime })}\leq \left( \frac{CS_{B}}{\min\left(T, R^{p}\right)^{ 1+\nu } }\right) ^{\frac{1}{\lambda \nu }}\left\Vert u\right\Vert _{L^{\lambda }(Q)},  \label{v31}\end{equation}%
where $C=C\left( p, \nu , \lambda \right) $.
\end{lemma}

\FRAME{ftbpFU}{2.1398in}{1.644in}{0pt}{\Qcb{Cylinders $Q$ and $Q^{\prime }$}%
}{\Qlb{pic2m}}{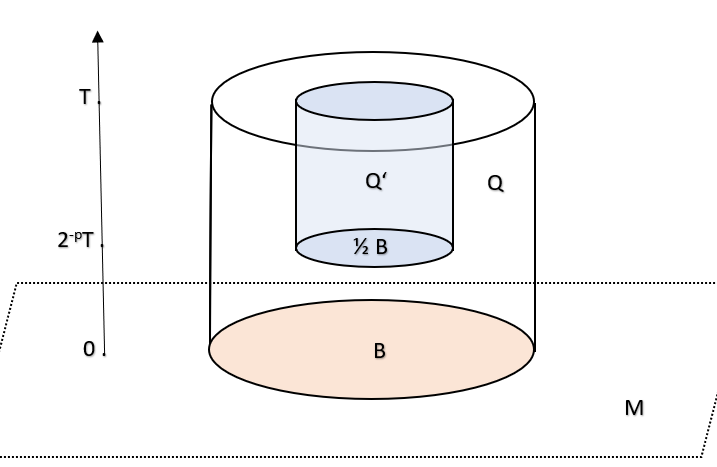}{\special{language "Scientific Word";type		"GRAPHIC";maintain-aspect-ratio TRUE;display "USEDEF";valid_file "F";width	2.1398in;height 1.644in;depth 0pt;original-width 10.0838in;original-height	7.7406in;cropleft "0";croptop "1";cropright "1";cropbottom "0";filename	'pic2.png';file-properties "XNPEU";}}

\begin{proof}
Let us first prove (\ref{v31}) for $\lambda \geq \max(p, \frac{p}{p-1})$.
Consider, for $k\geq 0$, sequences 
$$r_{k}=\left( \frac{1}{2}+2^{-(k+1)}\right) R\quad \textnormal{and} \quad t_{k}=\left(\left(\frac{1}{2}\right)^{p}-2^{-p(k+1)}\right)T$$ and set 
\begin{equation*}
B_{k}=B\left( x_{0},r_{k}\right) ,\ \ \ \ Q_{k}=B_{k}\times \left[ t_{k},T\right]
\end{equation*}%
so that 
\begin{equation*}
\,B_{0}=B,\ \ \ \ Q_{0}=Q\ \ \ \ \text{and\ \ \ }Q_{\infty}:=\lim_{k\rightarrow \infty }Q_{k}=Q^{\prime }
\end{equation*}%
(see Fig. \ref{pic3m}).\FRAME{dtbpFU}{2.5607in}{1.6588in}{0pt}{\Qcb{% 
Cylinders $Q_{k}$}}{\Qlb{pic3m}}{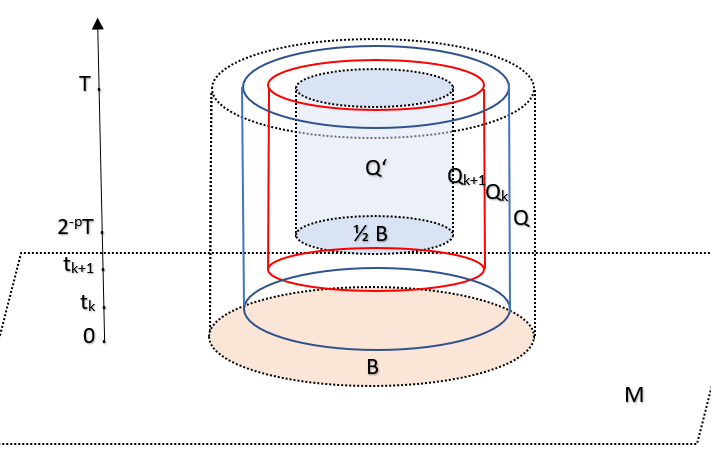}{\special{language"Scientific Word";type "GRAPHIC";maintain-aspect-ratio TRUE;display "USEDEF";valid_file "F";width 2.5607in;height 1.6588in;depth 0pt;original-width 11.078in;original-height 7.1615in;cropleft "0";croptop "1";cropright "1";cropbottom "0";filename 'pic3.png';file-properties "XNPEU";}}

Set also $\lambda _{k}=\lambda \left( 1+\nu \right) ^{k}$
and%
\begin{equation*}
J_{k}=\int_{Q_{k}}u^{\lambda _{k}}.
\end{equation*}%
By (\ref{vQ'QQ}) and using $r_{k}-r_{k+1}=2^{-\left( k+2\right) }R$ and $t_{k+1}-t_{k}=(2^{p}-1)2^{-p(k+2)}T$, we get
\begin{align*}
J_{k+1}& \leq \frac{CS_{B_{k}}\max\left(\lambda_{k}^{p-2}, \lambda _{k}^{\left( 2-p\right) \nu }\right)}{\min\left(t_{k+1}-t_{k}, \left(r_{k}-r_{k+1}\right)^{p}\right)^{ 1+\nu  }}J_{k}^{1+\nu } \\
& = \frac{CS_{B_{k}}2^{kp(1+\nu )}\max\left(\left( 1+\nu \right) ^{k\left( p-2\right)}\lambda^{p-2}, \left( 1+\nu \right) ^{k\left( 2-p\right) \nu}\lambda ^{\left( 2-p\right) \nu }\right)}{\min\left(T,R^{p} \right)^{1+\nu}}J_{k}^{1+\nu } \\
& \leq A^{k}\Theta ^{-1}J_{k}^{1+\nu },
\end{align*}%
where 
$A=2^{p\left( 1+\nu \right) }\max\left(\left( 1+\nu \right) ^{\left( p-2\right) _{+}}, \left( 1+\nu \right) ^{\nu\left( 2-p\right) _{+}}\right)$
and
\begin{equation*}
\Theta ^{-1}=\frac{CS_{B}}{\min\left(T,R^{p} \right)^{1+\nu}},
\end{equation*}%
where we have absorbed $\max\left(\lambda^{p-2}, \lambda ^{\nu\left( 2-p\right)}\right)$ into $C$. 

By Lemma \ref{LemJk} (see Appendix), we conclude that 
\begin{align*}
J_{k}& \leq \left( \left( A^{1/\nu }\Theta ^{-1}\right) ^{1/\nu
}J_{0}\right) ^{\left( 1+\nu \right) ^{k}}\left( A^{-1/\nu }\Theta \right)
^{1/\nu } \\
& =A^{\frac{(1+\nu )^{k}-1}{\nu ^{2}}}\Theta ^{-\frac{(1+\nu )^{k}-1}{\nu }%
}J_{0}^{(1+\nu )^{k}}.
\end{align*}%
It follows that%
\begin{equation*}
\left( \int_{Q_{k}}u^{\lambda _{k}}\right) ^{1/\lambda _{k}}\leq A^{\frac{%
1-(1+\nu )^{-k}}{\lambda \nu ^{2}}}\Theta ^{-\frac{1-(1+\nu )^{-k}}{\lambda
\nu }}\left( \int_{Q}u^{\lambda }\right) ^{1/\lambda }.
\end{equation*}%
As $k\rightarrow \infty $, we obtain%
\begin{align*}
\left\Vert u\right\Vert _{L^{\infty }(Q^{\prime })}& \leq A^{\frac{1}{\lambda\nu ^{2}}}\Theta ^{-\frac{1}{\lambda \nu }}\left\Vert u\right\Vert_{L^{\lambda }(Q)} \\
& = A^{\frac{1}{\lambda\nu ^{2}}}\left(\frac{CS_{B}}{\min\left(T,R^{p} \right)^{1+\nu}}\right) ^{\frac{1%
}{\lambda \nu }}\left\Vert u\right\Vert _{L^{\lambda }(Q)} \\
& =\left(\frac{CS_{B}}{\min\left(T,R^{p} \right)^{1+\nu}}\right) ^{\frac{1}{\lambda\nu }}%
\left\Vert u\right\Vert _{L^{\lambda }(Q)},
\end{align*}%
where $A^{1/\nu}$ was absorbed into $C$, which completes the proof in the case when $\lambda\geq \max(p, \frac{p}{p-1})$.

Now we prove (\ref{v31}) for any $\lambda >0$. Let $\lambda _{0}$ be such
that (\ref{v31}) is already known for $\lambda =\lambda _{0}$, for example, 
$\lambda _{0}=\max(p, pq)$, and let $\lambda <\lambda _{0}.$

%For simplicity of notation, for any set $E\subset M$, denote%
%\begin{equation*}
%E^{t}=E\times \left[ 0,t\right] .
%\end{equation*}%
%For example, 
%\begin{equation*}
%B^{t}(x,r)=B\left( x,r\right) \times \left[ 0,t\right] .
%\end{equation*}

%By the first part of the proof, we have, for any precompact ball $B$ of
%radius $r$, 
%\begin{equation*}
%	\left\Vert u\right\Vert _{L^{\infty }(\frac{1}{2}B^{t})}^{\lambda _{0}}\leq 
%	\frac{C}{\chi (B)r^{p}}\int_{B^{t}}u^{\sigma _{0}},
%\end{equation*}%
%where%
%\begin{equation*}
%	\chi (B)=\iota (B)\mu (B).
%\end{equation*}%
Consider, for $k\geq 0$, sequences
$r_{k}=\left(1-\frac{1}{2^{k+1}}\right)R,$
and $t_{k}=2^{-p(k+1)}T$
so that $$r_{0}=\frac{1}{2}r,\quad t_{0}=2^{-p}T\quad \textnormal{and}\quad r_{k}\uparrow R,\quad t_{k}\downarrow 0.$$
Set 
$B_{k}=B(x_{0},r_{k})$
and $\widetilde{Q}_{k}=B_{k}\times [t_{k}, T]$.
Denoting also $B=B(x_{0},R),$ we see that%
\begin{equation*}
	\frac{1}{2}B\subset B_{k}\subset B\ \ \text{and\ \ }B_{k}\uparrow B
\end{equation*}%
as $k\rightarrow \infty $ and thus $$\widetilde{Q}_{0}=Q^{\prime}\ \ \text{and\ \ }\widetilde{Q}_{k}\uparrow Q.$$ Set also 
$\rho _{k}=r_{k+1}-r_{k}=\frac{1}{2^{k+2}}R$.

For any point $(x, \tau)\in \widetilde{Q}_{k}$, let $s$ be such that $$\tau<s<\min\left(\tau+(1-2^{-p})t_{k}, T\right).$$
\FRAME{dtbpFU}{2.1662in}{1.8558in}{0pt}{\Qcb{Balls $B_{k}$ and $B(x, \rho_{k})$
}}{\Qlb{pic4m}}{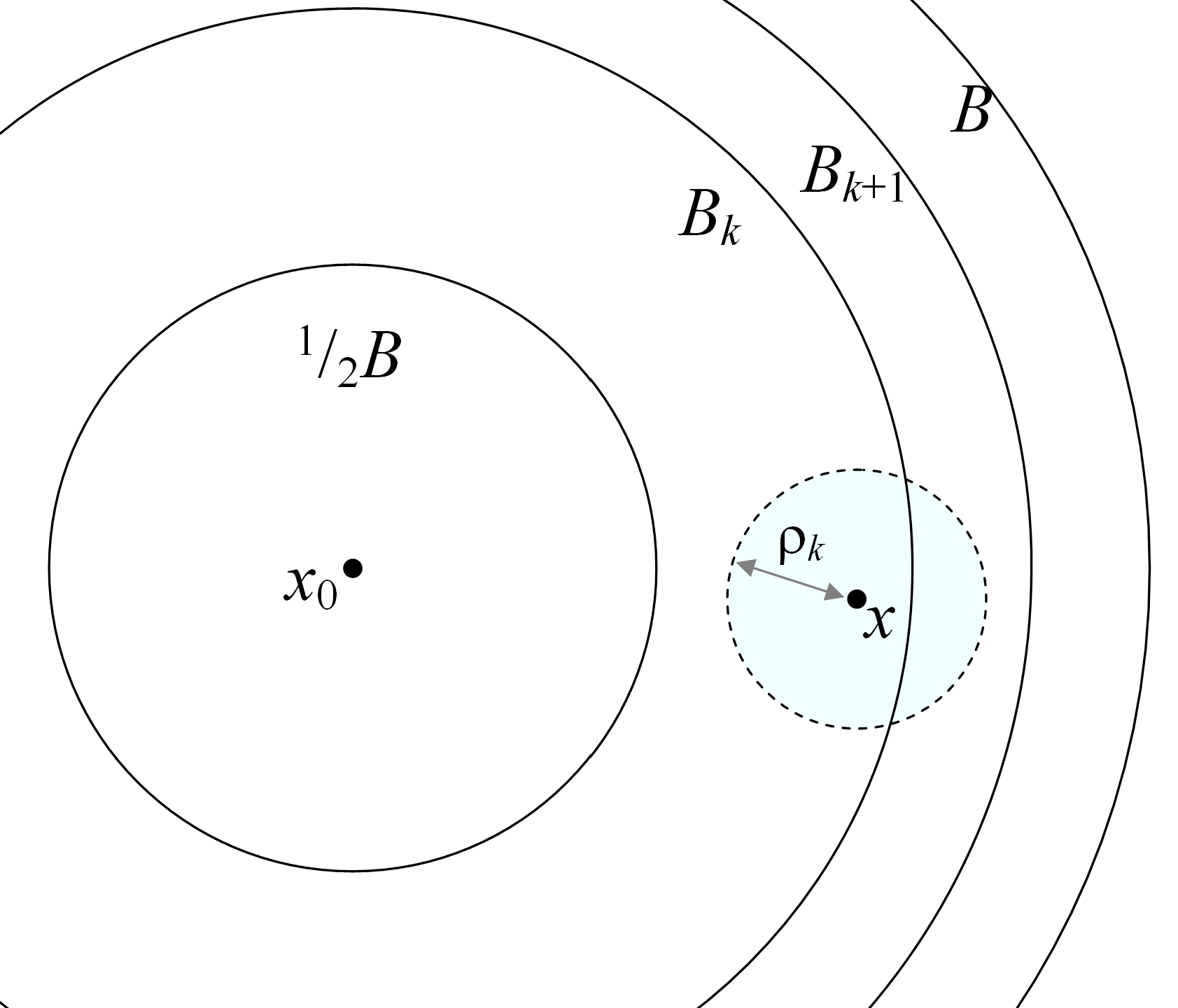}{\special{language"Scientific Word";type "GRAPHIC";maintain-aspect-ratio TRUE;display "USEDEF";valid_file "F";width 2.1662in;height 1.8558in;depth 0pt;original-width 7.8574in;original-height 6.7229in;cropleft "0";croptop "1";cropright "1";cropbottom "0";filename 'pic4.png';file-properties "XNPEU";}}
Then applying (\ref{v31}) from the first part of this proof in $%
B\left( x,\rho _{k}\right)\times [s-t_{k}, s)$, we obtain%
\begin{align*}
	\left\Vert u\right\Vert _{L^{\infty }(B(x,\frac{1}{2}\rho
		_{k})\times [s-(1-2^{-p})t_{k}, s))}^{\lambda _{0}}& \leq \left( \frac{CS_{B(x,\rho
			_{k})}}{\min\left(t_{k}, \rho_{k}^{p}\right)^{ 1+\nu } }\right) ^{\frac{1}{\nu }}\int_{B\left( x,\rho _{k}\right)\times [s-t_{k}, s)}u^{\lambda_{0}} \\
	& \leq \left( \frac{CS_{B(x,\rho
			_{k})}}{\min\left(t_{k}, \rho_{k}^{p}\right)^{ 1+\nu } }\right) ^{\frac{1}{\nu }}\left\Vert
	u\right\Vert _{L^{\infty }(B\left( x,\rho _{k}\right)\times [s-t_{k}, s))}^{\lambda _{0}-\lambda
	}\int_{B\left( x,\rho _{k}\right)\times [s-t_{k}, s)}u^{\lambda }.
\end{align*}%
Since 
$B\left( x,\rho _{k}\right) \subset B_{k+1}\subset B,$
we have by the monotonicity of $S$,
$S_{B(x,\rho_{k})}\leq S_{B}.$
Hence, we obtain 
\begin{equation*}
	\left\Vert u\right\Vert _{L^{\infty }(B(x,\frac{1}{2}\rho
		_{k})\times [s-(1-2^{-p})t_{k}, s))}^{\lambda _{0}}\leq \left( \frac{C2^{kp(1+\nu)}S_{B}}{\min\left(T, R^{p}\right)^{ 1+\nu } }\right) ^{\frac{1}{\nu }}\left\Vert u\right\Vert _{L^{\infty }(\widetilde{Q}_{k+1})}^{\lambda
		_{0}-\lambda }\int_{Q}u^{\lambda }.
\end{equation*}%
Covering $\widetilde{Q}_{k}$ by a sequence of sets $B(x,\frac{1}{2}\rho _{k})\times [s-(1-2^{-p})t_{k}, s))$ with $%
(x, \tau)\in \widetilde{Q}_{k}$, we obtain%
\begin{equation}
	\left\Vert u\right\Vert _{L^{\infty }(\widetilde{Q}_{k})}^{\lambda _{0}}\leq \left( \frac{C2^{kp(1+\nu)}S_{B}}{\min\left(T, R^{p}\right)^{ 1+\nu } }\right) ^{\frac{1}{\nu }}\left\Vert u\right\Vert
	_{L^{\infty }(\widetilde{Q}_{k+1})}^{\lambda _{0}-\lambda }\int_{Q}u^{\lambda }.
	\label{Bkt}
\end{equation}%
Setting 
$J_{k}=\left\Vert u\right\Vert _{L^{\infty }(\widetilde{Q}_{k})}^{-\left( \lambda
		_{0}-\lambda \right) },$
we rewrite (\ref{Bkt}) as follows:%
\begin{equation*}
	J_{k+1}\leq \frac{A^{k}}{\Theta }J_{k}^{\frac{\lambda _{0}}{\lambda
			_{0}-\lambda }}=\frac{A^{k}}{\Theta }J_{k}^{1+\omega },
\end{equation*}%
where $A=2^{\frac{p}{\nu}(1+\nu)}$, $\omega  =\frac{\lambda _{0}}{\lambda _{0}-\lambda }-1=\frac{\lambda }{%
	\lambda _{0}-\lambda }$ and 
\begin{align*}
	\Theta ^{-1}& =\left( \frac{CS_{B}}{\min\left(T, R^{p}\right)^{ 1+\nu } }\right) ^{\frac{1}{\nu }}\int_{Q}u^{\lambda }.
\end{align*}%
Applying Lemma \ref{LemJk}, we obtain%
\begin{equation*}
	J_{k}\leq \left( \frac{J_{0}}{\left( A^{-1/\omega }\Theta \right) ^{1/\omega
	}}\right) ^{\left( 1+\omega \right) ^{k}}\left( A^{-1/\omega }\Theta \right)
	^{1/\omega },
\end{equation*}%
that is, 
\begin{equation*}
	J_{0}\geq \left( A^{-1/\omega }\Theta \right) ^{1/\omega }\left( \left(
	A^{1/\omega }\Theta ^{-1}\right) ^{1/\omega }J_{k}\right) ^{\frac{1}{\left(
			1+\omega \right) ^{k}}}.
\end{equation*}%
Since 
$J_{k}\geq \left\Vert u\right\Vert _{L^{\infty }(\widetilde{Q}_{k})}^{-\left( \lambda
		_{0}-\lambda \right) }=:\func{const}>0,$
we see that%
\begin{equation*}
	\liminf_{k\rightarrow \infty }\left( \left( A^{1/\omega }\Theta ^{-1}\right)
	^{1/\omega }J_{k}\right) ^{\frac{1}{\left( 1+\omega \right) ^{k}}}\geq 1,
\end{equation*}%
whence%
\begin{equation*}
	J_{0}\geq \left( A^{-1/\omega }\Theta \right) ^{1/\omega }.
\end{equation*}%
It follows that $J_{0}^{-1}\leq A^{1/\omega ^{2}}\Theta ^{-1/\omega },$
that is,%
\begin{equation*}
	\left\Vert u\right\Vert _{L^{\infty }(\widetilde{Q}_{0})}^{\lambda _{0}-\lambda
	}\leq A^{1/\omega ^{2}}\left( \left( \frac{CS_{B}}{\min\left(T, R^{p}\right)^{ 1+\nu } }\right) ^{\frac{1}{\nu }}%
	\int_{Q}u^{\lambda }\right) ^{1/\omega },
\end{equation*}%
and finally
\begin{equation*}
	\left\Vert u\right\Vert _{L^{\infty }(Q^{\prime}
		)}\leq \left( \left( \frac{CS_{B}}{\min\left(T, R^{p}\right)^{ 1+\nu } }\right) ^{\frac{1}{\nu }}\int_{Q
	}u^{\lambda }\right) ^{1/\lambda },
\end{equation*}%
where $A^{1/\omega ^{2}}$ is absorbed into $C$ and this finishes the proof.
\end{proof}

\begin{remark}\normalfont
Let $\iota(B)$ be the normalized Sobolev constant defined by (\ref{defiota}), that is, $$\iota (B)=\dfrac{1}{\mu (B)}\left( \dfrac{r(B)^{p}}{S_{B}}\right) ^{\frac{1}{\nu}},$$ (see Subsection \ref{secMoser} for the meaning of this quantity).
If $R=T^{1/p}$, that is, $B=B(x_{0}, T^{1/p})$, then $$\left( \frac{CS_{B}}{T^{1+\nu}}\right)^{\frac{1}{\lambda\nu}}=\left(\frac{C\left(T^{1/p}\right)^{p}}{T^{1+\nu}\left(\iota(B)\mu(B)\right)^{\nu}}\right)^{\frac{1}{\lambda\nu}}=\left(\frac{C}{T\iota(B)\mu(B)}\right)^{\frac{1}{\lambda}}.$$ Hence, it follows from (\ref{v31}), that \begin{equation}\label{v3}\left\Vert u\right\Vert _{L^{\infty }(Q^{\prime })}\leq \left(\frac{C}{T\iota(B)\mu(B)}\right)^{\frac{1}{\lambda}}\left\Vert u\right\Vert _{L^{\lambda }(Q)}.\end{equation}
\end{remark}

\section{Sub-Gaussian upper bounds for subsolutions}
\label{secgauss}

The next lemma contains Lemma \ref{maintechlem} from the Introduction.
\begin{lemma}\label{Gaussup}
Let $u$ be a non-negative bounded subsolution of (\ref{dtv}) in $M\times [0, \infty)$ and set $A=\textnormal{supp}~u_{0}$. Fix some $\lambda \geq \max\left(p, \frac{p}{p-1}\right)$. Let $\rho>0$ be defined by (\ref{defofrho}) and let $\gamma$ be a regular function in $[0, \infty)$ satisfying for all $t>0$, \begin{equation}\label{condforgamma}\int_{A_{\rho}}u^{\lambda}(\cdot, t)\leq \frac{1}{\gamma(t)}.\end{equation} Then, for all $x\in M$ and all $T>0$, \begin{align}\nonumber||u(\cdot,T)||_{L^{\infty}\left(B(x,\frac{1}{2}T^{1/p})\right)}\leq&\left(\frac{C}{\iota(B\left( x,T^{1/p}\right))\mu(B\left( x,T^{1/p}\right))\gamma(cT)}\right)^{\frac{1}{\lambda}}\\&\label{uxTupper}\times\exp\left(-c\left(\frac{d(x,A)}{T^{1/p}}\right)^{\frac{p}{p-1}}\right),\end{align} where $\iota$ is as in (\ref{defiota}), $C=C(\Theta, \theta, p, \nu, \lambda)>0$ and $c=c(\theta, p, \lambda)>0$.
\end{lemma}
\begin{proof}
From Corollary \ref{CorD} we get, for all $t\simeq T$, $$\int_{B\left( x,T^{1/p}\right)}u^{\lambda}(\cdot, t)\leq\frac{C^{\prime}}{\gamma(t)} \exp\left(-c^{\prime}\left(\frac{d(x, A)}{t^{1/p}}\right)^{\frac{p}{p-1}}\right).$$ Setting $B=B\left( x,T^{1/p}\right)$ and integrating this inequality over $[(1-2^{-p})T, T]$ yields \begin{align*}\int_{B\times[(1-2^{-p})T, T]}u^{\lambda}\leq C^{\prime}\frac{T}{\gamma(cT)}\exp\left(-c^{\prime}\left(\frac{d(x,A)}{T^{1/p}}\right)^{\frac{p}{p-1}}\right).
\end{align*}
Combining this with (\ref{v3}), we obtain \begin{align*}||u(\cdot,T)||_{L^{\infty}\left(\frac{1}{2}B\right)}\leq\left(\frac{C}{\iota(B)\mu(B)\gamma(cT)}\right)^{\frac{1}{\lambda}}\exp\left(-\frac{c^{\prime}}{\lambda}\left(\frac{d(x,A)}{T^{1/p}}\right)^{\frac{p}{p-1}}\right),\end{align*} 
which proves (\ref{uxTupper}).
\end{proof}

\begin{definition}
We say that the manifold $M$ satisfies the \textit{doubling condition} if there exist positive constants $\alpha$ and $C$ such that, for all $x\in M$ and for all $0<r\leq R$, 
\begin{equation}\label{volumeregularity}
\frac{\mu(B(x, R))}{\mu(B(x, r))}\leq C\left(\frac{R}{r}\right)^{\alpha}.
\end{equation}	
\end{definition}

\begin{remark}
If $M$ satisfies the relative Faber-Krahn inequality, then it can be shown as in \cite{grigor1994heat}, which uses arguments from \cite{carron1996inegalites}, that the upper bound (\ref{volumeregularity}) holds with $\alpha=p/\nu$.
\end{remark}

The following result contains Theorem \ref{mainthmint}.

\begin{theorem}\label{corcurv}
Assume that $M$ satisfies the relative Faber-Krahn inequality. Let $u$ be a non-negative bounded subsolution of (\ref{dtv}) in $M\times [0, \infty)$ and set $A=\textnormal{supp}~u_{0}$. Then, for all $x\in M$ and all $T>0$, \begin{equation}\label{curvature}||u(\cdot,T)||_{L^{\infty}\left(B(x, \frac{1}{2}T^{1/p})\right)}\leq\frac{C}{\mu(B(x, T^{1/p}))}\exp\left(-c\left(\frac{d(x,A)}{T^{1/p}}\right)^{\frac{p}{p-1}}\right),\end{equation} where $c=c(p)>0$ and $C>0$ depends on $p, n$ and $||u_{0}||_{L^{1}(M)}$.
\end{theorem}

\begin{proof}
Fix some $\lambda\geq \max\left(p, \frac{p}{p-1}\right)$, $x\in M$ and $T>0$. Let us first consider the case when $d(x, A)>C^{\prime}T^{1/p}$ for some large constant $C^{\prime}>0$. Then it follows from (\ref{defofrho}) that $\rho=c^{\prime}d(x, A)$. Let us cover $A_{\rho}$ by a countable sequence of balls $B(x_{i}, \frac{1}{2}T^{1/p})$, with $x_{i}\in A_{\rho}$ and fix some $x_{0}\in A$. Then we obtain $d(x_{i}, A)\leq c^{\prime}d(x, A)$ and thus, \begin{equation}\label{ballincludedrho}B(x_{0}, \rho)\subset B(x_{i}, 2\rho).\end{equation}

From (\ref{v3}) we obtain $$||u(\cdot,T)||_{L^{\infty}\left(B(x_{i}, \frac{1}{2}T^{1/p})\right)}\leq \frac{C}{T\iota(B(x_{i}, T^{1/p}))\mu(B(x_{i}, T^{1/p}))}\int_{0}^{T}\int_{B(x_{i}, T^{1/p})}u,$$ where $C=C(p, \nu)$.
By assumption, we have that $\iota (B(x_{i}, T^{1/p}))\geq \func{const}>0$.
Hence, using Lemma \ref{monl1}, we obtain \begin{equation}\label{appof1mean}||u(\cdot,T)||_{L^{\infty}\left(B(x_{i}, \frac{1}{2}T^{1/p})\right)}\leq \frac{C}{\mu(B(x_{i}, T^{1/p}))}||u_{0}||_{L^{1}(M)}.\end{equation}
Let $\alpha=p/\nu$. Then it follows from (\ref{ballincludedrho}) and property (\ref{volumeregularity}), that \begin{align}
\frac{1}{\mu(B(x_{i}, T^{1/p}))}\leq \frac{\mu(B(x_{0}, \rho))}{\mu(B(x_{i}, T^{1/p}))\mu(B(x_{0}, T^{1/p}))}&\leq \frac{\mu(B(x_{i}, 2\rho))}{\mu(B(x_{i}, T^{1/p}))\mu(B(x_{0}, T^{1/p}))}\nonumber\\&\leq C\left(\frac{\rho}{T^{1/p}}\right)^{\alpha}\frac{1}{\mu(B(x_{0}, T^{1/p}))}\label{upperball}.
\end{align}
Combining (\ref{upperball}) with (\ref{appof1mean}), we deduce, for all $i$, $$||u(\cdot,T)||_{L^{\infty}\left(B(x_{i}, \frac{1}{2}T^{1/p})\right)}\leq C\left(\frac{\rho}{T^{1/p}}\right)^{\alpha}\frac{1}{\mu(B(x_{0}, T^{1/p}))},$$ and thus, $$||u(\cdot,T)||_{L^{\infty}\left(A_{\rho}\right)}\leq C\left(\frac{\rho}{T^{1/p}}\right)^{\alpha}\frac{1}{\mu(B(x_{0}, T^{1/p}))}.$$
It follows that \begin{equation}\label{decayofußlamb}\int_{A_{\rho}}u^{\lambda}(\cdot, T)\leq C\left(\frac{\rho}{T^{1/p}}\right)^{\alpha\lambda}\frac{\mu(B(x_{0}, \rho))}{\mu(B(x_{0}, T^{1/p}))^{\lambda}}.\end{equation}
In particular, this implies that the condition (\ref{condforgamma}) is satisfied with $\gamma$ defined by $$\frac{1}{\gamma(t)}=C\left(\frac{\rho}{t^{1/p}}\right)^{\alpha\lambda}\frac{\mu(B(x_{0}, \rho))}{\mu(B(x_{0}, t^{1/p}))^{\lambda}}.$$ Note that $\gamma$ defined in this way is a regular function by the doubling condition of $\mu$.

Substituting this into (\ref{uxTupper}), we get \begin{align*}||u(\cdot,T)||_{L^{\infty}\left(B(x, \frac{1}{2}T^{1/p})\right)}&\leq C\left(\frac{\rho}{T^{1/p}}\right)^{\alpha}\left(\frac{\mu(B(x_{0}, \rho))}{\mu(B(x_{0}, T^{1/p}))^{\lambda}\mu(B(x, T^{1/p}))}\right)^{1/\lambda}\\&~\times\exp\left(-c\left(\frac{d(x,A)}{T^{1/p}}\right)^{\frac{p}{p-1}}\right).\end{align*}
Using again the doubling condition of $\mu$, we obtain \begin{align*}
\left(\frac{\mu(B(x, C_{0}\rho))}{\mu(B(x_{0}, T^{1/p}))^{\lambda}\mu(B(x, T^{1/p}))}\right)^{1/\lambda}&\leq C\left(\frac{\rho}{T^{1/p}}\right)^{\alpha/\lambda}\frac{1}{\mu(B(x_{0}, T^{1/p}))}\\&\leq C\left(\frac{\rho}{T^{1/p}}\right)^{\alpha+\alpha/\lambda}\frac{1}{\mu(B(x, T^{1/p}))}.\end{align*}
Therefore, $$||u(\cdot,T)||_{L^{\infty}\left(B(x, \frac{1}{2}T^{1/p})\right)}\leq C\left(\frac{\rho}{T^{1/p}}\right)^{2\alpha+\alpha/\lambda}\frac{1}{\mu(B(x, T^{1/p}))}\exp\left(-c\left(\frac{d(x,A)}{T^{1/p}}\right)^{\frac{p}{p-1}}\right).$$
Since $\rho=c^{\prime}d(x, A)$, the term $\left(\frac{\rho}{T^{1/p}}\right)^{2\alpha+\alpha/\lambda}$ cancels with the term $\exp\left(-c\left(\frac{d(x,A)}{T^{1/p}}\right)^{\frac{p}{p-1}}\right)$ so that (\ref{curvature}) holds in the case $d(x, A)>C^{\prime}T^{1/p}$. In the case $d(x, A)\leq C^{\prime}T^{1/p}$, we have $\rho=c^{\prime}C^{\prime}T^{1/p}$ and (\ref{curvature}) follows from using the same arguments. This finishes the proof.
\end{proof}

\begin{remark}
Let $M=\mathbb{R}^{n}$. Then $\mu(B(x, T^{1/p}))\geq c T^{n/p}$ and (\ref{curvature}) gives the same upper bound (up to constants) as in (\ref{barenborder}). Hence, (\ref{curvature}) gives a sharp upper bound in $\mathbb{R}^{n}$. Moreover, the model manifold constructed in Proposition \ref{supersolmodel} satisfies the volume doubling property and the Poincar\'{e} inequality, and in particular, also the relative Faber-Krahn inequality (see Proposition 4.10 in \cite{grigor2005stability}). Hence, we do not only obtain a sharp upper bound in $\mathbb{R}^{n}$, but also in the class of manifolds constructed in Proposition \ref{supersolmodel}.
\end{remark}

The next result contains Theorem \ref{upperCHint}.

\begin{theorem}\label{upperCH}
Let $M$ satisfy the uniform Sobolev inequality and $\textnormal{dim}_{M}=n$. Let $u$ be a non-negative bounded subsolution of (\ref{dtv}) in $M\times [0, \infty)$ and set $A=\textnormal{supp}~u_{0}$. Then, for all $x\in M$ and all large enough $T>0$, \begin{equation}\label{carthar}||u(\cdot,T)||_{L^{\infty}\left(B(x,\frac{1}{2}T^{1/p})\right)}\leq\frac{C}{T^{n/p}}\exp\left(-c\left(\frac{d(x,A)}{T^{1/p}}\right)^{\frac{p}{p-1}}\right),\end{equation} where $c=c(p)>0$ and $C$ depends on $p, n$ and $||u_{0}||_{L^{1}(M)}$.
\end{theorem}
\begin{proof}
Fix $\lambda\geq \max\left(p, \frac{p}{p-1}\right)$. Let us first show the following \textit{Nash type inequality}. For non-negative $w\in W^{1, p}(M)\cap L^{p/\lambda}(M)\setminus\{0\}$, we have \begin{equation}\label{Nashtype}\int_{M}\left|\nabla w\right|^{p}\geq c\left(\int_{M}w^{p}\right)^{1+\frac{\nu}{(\lambda-1)}}\left(\int_{M}w^{p/\lambda}\right)^{-\frac{\nu\lambda}{(\lambda-1)}},\end{equation} where $c=c(p, n)>0$ and $\nu$ is given by (\ref{nu}).

By assumption we have for any non-negative $w\in W^{1, p}(M)$, \begin{equation}\label{sobolevM}\left( \int_{M}w^{p\kappa }\right) ^{1/\kappa }\leq C\int_{M}\left\vert
\nabla w\right\vert ^{p},\end{equation} where $C=C(p, n)>0$ and $\kappa$ is given by (\ref{k}). Note that by Hölder's inequality with exponent $\kappa\left(1+\frac{\nu}{(\lambda-1)}\right)$ and its Hölder conjugate $\omega=\frac{(\lambda-1)+\nu}{\lambda \nu}$, we have \begin{align*}\int_{M}w^{p}&=\int_{M}\left(w^{p\kappa}\right)^{\frac{1}{\kappa\left(1+\frac{\nu}{(\lambda-1)}\right)}}w^{p-\frac{p}{1+\frac{\nu}{(\lambda-1)}}}\\&\leq\left(\int_{M}w^{p\kappa}\right)^{\frac{1}{\kappa\left(1+\frac{\nu}{(\lambda-1)}\right)}}\left(\int_{M}\left(w^{p-\frac{p}{1+\frac{\nu}{(\lambda-1)}}}\right)^{\omega}\right)^{1/\omega},
\end{align*}
whence (\ref{sobolevM}) yields \begin{align*} \left(\int_{M}w^{p}\right)^{1+\frac{\nu}{(\lambda-1)}}&\leq \left(\int_{M}w^{p\kappa}\right)^{\frac{1}{\kappa}}\left(\int_{M}w^{p/\lambda}\right)^{\frac{\nu\lambda}{(\lambda-1)}}\\&\leq C\int_{M}\left\vert
\nabla w\right\vert^{p}\left(\int_{M}w^{p/\lambda}\right)^{\frac{\nu\lambda}{(\lambda-1)}},\end{align*} which proves (\ref{Nashtype}).

Setting $$\Phi(t)=\int_{M }u^{\lambda }(\cdot, t)$$ we see that by the Caccioppoli type inequality (\ref{veta1}), by choosing $\eta$ as a bump function with support in some geodesic ball $B$ and then sending $B\to M$, for $t_{1}<t_{2}$, $$\Phi(t_{2})-\Phi(t_{1})\leq -c_{1}%
\int_{M\times[t_{1}, t_{2}]}\left\vert \nabla \left( u^{\lambda/p } \right) \right\vert
^{p},$$ so that, in particular, $\Phi$ is monotone decreasing. Therefore, $\Phi(t)$ is differentiable for almost every $t>0$. Applying (\ref{Nashtype}) to $w(\cdot, t)=u^{\lambda/p}(\cdot, t)\in W^{1, p}(M)\cap L^{p/\lambda}(M)$ (see (\ref{valpha})) and Lemma \ref{monl1}, we get \begin{align*}
\Phi^{\prime}(t)&\leq -c_{1}%
\int_{M}\left\vert \nabla \left( u^{\lambda/p }(\cdot, t) \right) \right\vert^{p}\\&\leq -c\left(\int_{M}u^{\lambda}(\cdot, t)\right)^{1+\frac{\nu}{(\lambda-1)}}\left(\int_{M}u(\cdot, t)\right)^{-\frac{\nu\lambda}{(\lambda-1)}}\\&\leq-c\Phi(t)^{1+\frac{\nu}{(\lambda-1)}}||u_{0}||_{L^{1}(M)}^{-\frac{\nu\lambda}{(\lambda-1)}}.
\end{align*}
Solving this inequality we get, for almost every $t>0$, $$\Phi(t)\leq \frac{C}{t^{\frac{\lambda-1}{\nu}}}.$$ Noticing that the left hand side of (\ref{condforgamma}) is continuous in $t$, (\ref{condforgamma}) holds true for all $t>0$ with \begin{equation}\label{uppergamma}\gamma(t)=C^{-1}t^{\frac{\lambda-1}{\nu}}.\end{equation}
By (\ref{defiota}), we have $$\iota(B(x, T^{1/p}))=\frac{1}{\mu(B(x, T^{1/p}))}\left(\frac{T}{S_{B}}\right)^{1/\nu}\geq \frac{cT^{1/\nu}}{\mu(B(x, T^{1/p}))}.$$
Therefore, substituting this and (\ref{uppergamma}) into (\ref{uxTupper}), we deduce $$||u(\cdot,T)||_{L^{\infty}\left(B(x,\frac{1}{2}T^{1/p})\right)}\leq\left(\frac{C}{T^{ \lambda/\nu}}\right)^{\frac{1}{\lambda}}\exp\left(-c\left(\frac{d(x,A)}{T^{1/p}}\right)^{\frac{p}{p-1}}\right),$$ which proves (\ref{carthar}), since $\nu\leq \frac{p}{n}$ by (\ref{nu}).
\end{proof}

\section{Appendix}

\label{appendixla}

\subsection{Radial solution on polynomial models}

Let $M$ be a model manifold, that is $M=(0, +\infty)\times \mathbb{S}^{n-1}$ as topological spaces and $M$ is equipped with the Riemannian metric $ds^{2}$ given by \begin{equation*}ds^{2}=dr^{2}+\psi^{2}(r)d\theta^{2},\end{equation*} where $\psi(r)$ is a smooth positive function on $(0, +\infty)$ and $d\theta^{2}$ is the standard Riemannian metric on $\mathbb{S}^{n-1}$. We define $S(r)=\psi^{n-1}(r)$, which is called the profile of the model manifold.

We search for solutions $u$ of the Trudinger equation (\ref{evoeq1}).
Let $u(x, t)=u(r, t)$, that is, function $u$ depends only on the polar radius $r$ and time $t$. Assume also that $\partial_{r}u\leq 0$, then $$\Delta_{p}u=-\frac{1}{S}\partial _{r}\left( S\left( -\partial
_{r}u\right) ^{p-1}\right)$$ so that (\ref{dtv}) becomes \begin{equation}\label{evoeqmodel}\partial _{t}u=-\frac{1}{S}\partial _{r}\left( S\left( -\partial
	_{r}u^{\frac{1}{p-1}}\right) ^{p-1}\right).\end{equation}

\begin{proposition}\label{supersolmodel}
	Assume that, for some $\alpha \in (0,n]$ and all $r\geq r_{0}$,
	\begin{equation*}
		S\left( r\right) =Cr^{\alpha -1}.
	\end{equation*}
	Then the following function is a non-negative
	solution of (\ref{dtv}) in $M\setminus B_{r_{0}}\times \mathbb{R}_{+}$: 
	\begin{equation}
		u\left( x,t\right) =\frac{1}{t^{\alpha /p }}\exp\left( -\zeta \left( \frac{%
			r}{t^{1/p }}\right) ^{\frac{p}{p-1}}\right), \label{ua}
	\end{equation}%
	where $\zeta=(p-1)^{2}p^{-\frac{p}{p-1}}$.
\end{proposition}

\begin{proof}
	By (\ref{evoeqmodel}) the equation (\ref{dtv}) for $u$ becomes for $r>r_{0}$, 
	\begin{equation}\label{uS'S}
		\partial _{t}u=-\frac{1}{r^{\alpha -1}}\partial _{r}\left( r^{\alpha
			-1}\left( -\partial _{r}u^{\frac{1}{p-1}}\right) ^{p-1}\right).
	\end{equation}%
	We search for a solution of the form%
	\begin{equation*}
		u\left( x,t\right) =t^{a}f\left( rt^{b}\right) \ \ \text{for large }r\text{,}
	\end{equation*}%
	where $f$ is a decreasing function. Let us require in addition that the
	solution $u\left( \cdot ,t\right) $ has bounded $L^{1}$-norm. One can show that for that we need to require that $a=\alpha b$.
	Using the variable $s=rt^{b}$, we obtain that 
	(\ref{uS'S}) is equivalent to%
	\begin{align*}
		\frac{bt^{a-1}}{s^{\alpha -1}}\left( s^{\alpha }f\left( s\right) \right)
		^{\prime }=-\frac{t^{\left( \frac{a}{p-1}+b\right) \left( p-1\right) }}{(p-1)^{p-1}s^{\alpha -1}}%
		t^{b}\partial _{s}\left( s^{\alpha -1}\left( -f(s)^{\frac{1}{p-1}-1}f^{\prime
		}(s)\right) ^{p-1}\right) .
	\end{align*}%
	Let us also require that%
	\begin{equation*}
		\left( \frac{a}{p-1}+b\right) \left( p-1\right) +b=a-1,
	\end{equation*}%
	which yields $b=-\dfrac{1}{p}<0$.
	Under the above choice of $a$ and $b$, the powers of $t$ and $s$ in the
	above equation cancel out, and we obtain since $b<0$,
	\begin{equation}
		\frac{f^{\prime}}{f}=-(p-1)\left( \frac{1}{p} s\right) ^{\frac{1}{p-1}}.  \label{ODE1}
	\end{equation}%
	Integration of (\ref{ODE1}) yields
	\begin{equation*}
		f\left( s\right) =\exp\left( -\zeta s^{\frac{p}{p-1}}\right)
	\end{equation*}%
	where
	$\zeta =\frac{(p-1)^{2}}{p^{\frac{p}{p-1}}}$ which proves the claim.
\end{proof}

\subsection{An auxiliary lemma}

\begin{lemma}[\cite{grigor2023finite}]
\label{LemJk}Let a sequence $\left\{ J_{k}\right\} _{k=0}^{\infty }$ of
non-negative reals satisfy%
\begin{equation}
J_{k+1}\leq \frac{A^{k}}{\Theta }J_{k}^{1+\omega }\ \ \text{for all }k\geq 0.
\label{Jk+1}
\end{equation}%
where $A,\Theta ,\omega >0.$ Then, for all $k\geq 0$,%
\begin{equation}
J_{k}\leq \left( \frac{J_{0}}{\left( A^{-1/\omega }\Theta \right) ^{1/\omega
}}\right) ^{\left( 1+\omega \right) ^{k}}\left( A^{-k-1/\omega }\Theta
\right) ^{1/\omega }.  \label{Jk<}
\end{equation}
In particular, if 
\begin{equation}
\Theta \geq A^{1/\omega }J_{0}^{\omega },  \label{T>}
\end{equation}%
then, for all $k\geq 0,$%
\begin{equation}
J_{k}\leq A^{-k/\omega }J_{0}.  \label{JkJ0}
\end{equation}
\end{lemma}

\bibliographystyle{abbrv}
\bibliography{librarycacc}

\begin{thebibliography}{10}

\bibitem{andreucci2021asymptotic}
D.~Andreucci and A.~F. Tedeev.
\newblock Asymptotic properties of solutions to the {C}auchy problem for
  degenerate parabolic equations with inhomogeneous density on manifolds.
\newblock {\em Milan Journal of Mathematics}, 89(2):295--327, 2021.

\bibitem{andreucci2021extinction}
D.~Andreucci and A.~F. Tedeev.
\newblock Extinction in a finite time for parabolic equations of fast diffusion
  type on manifolds.
\newblock In {\em Operator theory and differential equations}, pages 1--6.
  Springer, 2021.

\bibitem{Aronson1967BoundsFT}
D.~G. Aronson.
\newblock Bounds for the fundamental solution of a parabolic equation.
\newblock {\em Bulletin of the American Mathematical Society}, 73:890--896,
  1967.

\bibitem{aronson1968non}
D.~G. Aronson.
\newblock Non-negative solutions of linear parabolic equations.
\newblock {\em Annali della Scuola Normale Superiore di Pisa-Classe di
  Scienze}, 22(4):607--694, 1968.

\bibitem{barenblatt1952self}
G.~I. Barenblatt.
\newblock On self-similar motions of a compressible fluid in a porous medium.
\newblock {\em Akad. Nauk SSSR. Prikl. Mat. Meh}, 16(6):679--698, 1952.

\bibitem{barlow2006diffusions}
M.~T. Barlow.
\newblock Diffusions on fractals.
\newblock {\em Lectures on Probability Theory and Statistics: Ecole d'Et{\'e}
  de Probabilit{\'e}s de Saint-Flour XXV—1995}, pages 1--121, 2006.

\bibitem{bonforte2005asymptotics}
M.~Bonforte and G.~Grillo.
\newblock Asymptotics of the porous media equation via {S}obolev inequalities.
\newblock {\em Journal of Functional Analysis}, 225(1):33--62, 2005.

\bibitem{bonforte2008fast}
M.~Bonforte, G.~Grillo, and J.~L. Vazquez.
\newblock Fast diffusion flow on manifolds of nonpositive curvature.
\newblock {\em Journal of Evolution Equations}, 8:99--128, 2008.

\bibitem{Buser}
P.~Buser.
\newblock A note on the isoperimetric constant.
\newblock {\em Ann. Sci. Ecole Norm. Sup.}, 15:213--230, 1982.

\bibitem{carron1996inegalites}
G.~Carron.
\newblock {In{\'e}galit{\'e}s isop{\'e}rim{\'e}triques de {F}aber-{K}rahn et
  cons{\'e}quences}.
\newblock In {\em Actes de la table ronde de g{\'e}om{\'e}trie
  diff{\'e}rentielle (Luminy, 1992)”, Collection SMF S{\'e}minaires et
  Congres}, volume~1, pages 205--232, 1996.

\bibitem{Coulhon1997}
T.~Coulhon and A.~Grigor'yan.
\newblock On-diagonal lower bounds for heat kernels and {M}arkov chains.
\newblock {\em Duke Math. J.}, 89:133--199, 1997.

\bibitem{coulhon1998random}
T.~Coulhon and A.~Grigoryan.
\newblock Random walks on graphs with regular volume growth.
\newblock {\em Geometric and Functional Analysis}, 8(4):656--701, 1998.

\bibitem{davies1992heat}
E.~B. Davies.
\newblock Heat kernel bounds, conservation of probability and the feller
  property.
\newblock {\em Journal d’Analyse Math{\'e}matique}, 58(1):99--119, 1992.

\bibitem{de2022wasserstein}
N.~De~Ponti, M.~Muratori, and C.~Orrieri.
\newblock Wasserstein stability of porous medium-type equations on manifolds
  with {R}icci curvature bounded below.
\newblock {\em Journal of Functional Analysis}, 283(9):109661, 2022.

\bibitem{dekkers2005finite}
S.~Dekkers.
\newblock Finite propagation speed for solutions of the parabolic $p$-laplace
  equation on manifolds.
\newblock {\em Communications in Analysis and Geometry}, 13(4):741--768, 2005.

\bibitem{del2004nonlinear}
M.~Del~Pino, J.~Dolbeault, and I.~Gentil.
\newblock Nonlinear diffusions, hypercontractivity and the optimal lp-euclidean
  logarithmic sobolev inequality.
\newblock {\em Journal of Mathematical Analysis and Applications},
  293(2):375--388, 2004.

\bibitem{dibenedetto2011harnack}
E.~DiBenedetto, U.~P. Gianazza, and V.~Vespri.
\newblock {\em Harnack's inequality for degenerate and singular parabolic
  equations}.
\newblock Springer Science \& Business Media, 2011.

\bibitem{grigor}
A.~Grigor'yan.
\newblock The heat equation on non-compact {R}iemannian manifolds.
\newblock {\em Math. USSR Sb.}, 72:47--77, 1992.

\bibitem{grigor1994heat}
A.~Grigor'yan.
\newblock Heat kernel upper bounds on a complete non-compact manifold.
\newblock {\em Revista Matem{\'a}tica Iberoamericana}, 10(2):395--452, 1994.

\bibitem{Grigoryan2012}
A.~Grigor'yan.
\newblock {\em Heat Kernel and Analysis on Manifolds}.
\newblock American Mathematical Society, nov 2012.

\bibitem{grigor2005stability}
A.~Grigor'yan and L.~Saloff-Coste.
\newblock Stability results for harnack inequalities.
\newblock In {\em Annales de l'institut Fourier}, volume~55, pages 825--890,
  2005.

\bibitem{grigor2009heat}
A.~Grigor'yan and L.~Saloff-Coste.
\newblock Heat kernel on manifolds with ends.
\newblock In {\em Annales de l'institut Fourier}, volume~59, pages 1917--1997,
  2009.

\bibitem{grigor2016surgery}
A.~Grigor'yan and L.~Saloff-Coste.
\newblock Surgery of the {F}aber--{K}rahn inequality and applications to heat
  kernel bounds.
\newblock {\em Nonlinear Analysis}, 131:243--272, 2016.

\bibitem{grigor1997gaussian}
A.~Grigor’yan.
\newblock Gaussian upper bounds for the heat kernel on arbitrary manifolds.
\newblock {\em J. Diff. Geom}, 45(1):33--52, 1997.

\bibitem{grigor1999estimates}
A.~Grigor’yan.
\newblock Estimates of heat kernels on {R}iemannian manifolds.
\newblock {\em London Math. Soc. Lecture Note Ser}, 273:140--225, 1999.

\bibitem{grigor2022volume}
A.~Grigor’yan and P.~S{\"u}rig.
\newblock Volume growth and on-diagonal heat kernel bounds on {R}iemannian
  manifolds with an end.
\newblock {\em Potential Analysis}, pages 1--33, 2022.

\bibitem{grigor2023finite}
A.~Grigor’yan and P.~S{\"u}rig.
\newblock Finite propagation speed for {L}eibenson’s equation on {R}iemannian
  manifolds.
\newblock {\em to appear in Comm. Anal. Geom.}, 2023.

\bibitem{Grigor’yan2024}
A.~Grigor’yan and P.~S{\"u}rig.
\newblock Sharp propagation rate for {L}eibenson’s equation on {R}iemannian
  manifolds.
\newblock {\em preprint}, 2024.

\bibitem{grillo2006super}
G.~Grillo and M.~Bonforte.
\newblock Super and ultracontractive bounds for doubly nonlinear evolution
  equations.
\newblock {\em Revista matem{\'a}tica iberoamericana}, 22(1):111--129, 2006.

\bibitem{grillo2016smoothing}
G.~Grillo and M.~Muratori.
\newblock Smoothing effects for the porous medium equation on
  {C}artan--{H}adamard manifolds.
\newblock {\em Nonlinear Analysis}, 131:346--362, 2016.

\bibitem{grillo2018porous}
G.~Grillo, M.~Muratori, and F.~Punzo.
\newblock The porous medium equation with measure data on negatively curved
  {R}iemannian manifolds.
\newblock {\em Journal of the European Mathematical Society},
  20(11):2769--2812, 2018.

\bibitem{hoffman1974sobolev}
D.~Hoffman and J.~Spruck.
\newblock Sobolev and isoperimetric inequalities for {R}iemannian submanifolds.
\newblock {\em Communications on Pure and Applied Mathematics}, 27(6):715--727,
  1974.

\bibitem{ishige1996existence}
K.~Ishige.
\newblock On the existence of solutions of the cauchy problem for a doubly
  nonlinear parabolic equation.
\newblock {\em SIAM Journal on Mathematical Analysis}, 27(5):1235--1260, 1996.

\bibitem{ivanov1997regularity}
A.~V. Ivanov.
\newblock Regularity for doubly nonlinear parabolic equations.
\newblock {\em Journal of Mathematical Sciences}, 83(1):22--37, 1997.

\bibitem{ladyzhenskaya1968linear}
O.~Ladyzhenskaya, V.~Solonnikov, and N.~Ural’tseva.
\newblock Linear and quasilinear equations of parabolic type, transl. math.
\newblock {\em Monographs, Amer. Math. Soc}, 23, 1968.

\bibitem{Li1986OnTP}
P.~Li and S.-T. Yau.
\newblock On the parabolic kernel of the {S}chr{\"o}dinger operator.
\newblock {\em Acta Mathematica}, 156:153--201, 1986.

\bibitem{Moser}
J.~Moser.
\newblock Harnack inequality for parabolic differential equations.
\newblock {\em Comm. Pure Appl. Math.}, 17:101--134, 1964.

\bibitem{raviart1970resolution}
P.-A. Raviart.
\newblock Sur la r{\'e}solution de certaines {\'e}quations paraboliques non
  lin{\'e}aires.
\newblock {\em Journal of Functional Analysis}, 5(2):299--328, 1970.

\bibitem{Saloff}
L.~Saloff-Coste.
\newblock {\em Aspects of Sobolev-type inequalities}.
\newblock LMS Lecture Notes Series, vol. 289. Cambridge Univ. Press, 2002.

\bibitem{sturm2017existence}
S.~Sturm.
\newblock Existence of weak solutions of doubly nonlinear parabolic equations.
\newblock {\em Journal of Mathematical Analysis and Applications},
  455(1):842--863, 2017.

\bibitem{trudinger1968pointwise}
N.~S. Trudinger.
\newblock Pointwise estimates and quasilinear parabolic equations.
\newblock {\em Communications on Pure and Applied Mathematics}, 21(3):205--226,
  1968.

\bibitem{vazquez2015fundamental}
J.~L. V{\'a}zquez.
\newblock Fundamental solution and long time behavior of the porous medium
  equation in {H}yperbolic space.
\newblock {\em Journal de Math{\'e}matiques Pures et Appliqu{\'e}es},
  104(3):454--484, 2015.

\end{thebibliography}

\emph{Universit\"{a}t Bielefeld, Fakult\"{a}t f\"{u}r Mathematik, Postfach
	100131, D-33501, Bielefeld, Germany}

\texttt{philipp.suerig@uni-bielefeld.de}

\end{document}